\documentclass[11pt]{amsart}
\usepackage[text={160mm, 230mm},centering]{geometry}
\usepackage{float}
\usepackage{appendix}
\usepackage{graphicx}
\usepackage{chngpage}
\usepackage{multirow}
\usepackage[colorlinks=true, citecolor=blue, urlcolor=blue, linkcolor=blue]{hyperref}
\usepackage[all,cmtip]{xy}
\usepackage{amssymb}
\usepackage{amsmath}
\usepackage{longtable}
\usepackage{pstricks}
\usepackage{color}
\usepackage[all]{xy}
\theoremstyle{plain}
\newtheorem{thm}{Theorem}[section]
\newtheorem{theorem}[thm]{Theorem}

\newtheorem{lemma}[thm]{Lemma}
\newtheorem{corollary}[thm]{Corollary}
\newtheorem{proposition}[thm]{Proposition}
\theoremstyle{definition}
\newtheorem{remark}[thm]{Remark}

\newtheorem{definition}[thm]{Definition}

\newtheorem{example}[thm]{Example}

\newtheorem{strategy}[thm]{Strategy}
\numberwithin{equation}{section}

\newcommand{\Aut}{{\rm Aut}}

\newcommand{\Diag}{{\rm diag}}
\newcommand{\Rank}{{\rm rk}}

\newcommand{\GL}{{\rm GL}}
\newcommand{\PGL}{{\rm PGL}}
\newcommand{\PSL}{{\rm PSL}}

\newcommand{\C}{{\mathbb C}}

\renewcommand{\P}{{\mathbb P}}

\makeatletter
\@namedef{subjclassname@2020}{%
  \textup{2020} Mathematics Subject Classification}
\makeatother

\begin{document}

	\title{Automorphism groups of cubic fivefolds and fourfolds}
	
	\author{Song Yang, Xun Yu \and Zigang Zhu}
	
	\address{Center for Applied Mathematics and KL-AAGDM, Tianjin University, Weijin Road 92, Tianjin 300072, P.R. China}
	
	\email{syangmath@tju.edu.cn, xunyu@tju.edu.cn, zhzg0313@tju.edu.cn}

	\begin{abstract}
		In this paper, we introduce notions of partitionability and characteristic sets of homogeneous polynomials and give a complete classification of groups faithfully acting on smooth cubic fivefolds.  Specifically, we prove that there exist 20 maximal ones among all such groups. As an application, we classify all possible subgroups of the automorphism groups of smooth cubic fourfolds.
	\end{abstract}
	
	
	\vspace*{-.4in}
	\maketitle
	\setcounter{tocdepth}{1}
    \vspace*{-.4in}
	\tableofcontents
    \vspace*{-.6in}
		
	
	\section{Introduction}\label{Intro}
	In this paper, we study the automorphism groups of smooth cubic hypersurfaces $X$ in the projective space over the complex number field $\C$. Such hypersurfaces are an important class of projective varieties in algebraic geometry. For instance, cubic threefolds are unirational but not rational \cite{CG72}, and the development of topics related to smooth cubic hypersurfaces can be found in \cite{Huy}. The study of their automorphism groups $\Aut(X)$ has a long and rich history, see \cite{Se42}, \cite{Ad78}, \cite{Ho97}, \cite{Ro09}, \cite{GL11}, \cite{Do12}, \cite{Pro12}, \cite{GL13}, \cite{Mo13}, \cite{BCS16}, \cite{Fu16}, \cite{DD19}, \cite{HM19}, \cite{WY20}, \cite{LZ22}, \cite{Zh22}, \cite{AKPW23}, \cite{GLM23}, etc. All possible subgroups of $\Aut(X)$  have been classified for cubic surfaces (see \cite{Se42}, \cite{Ho97}, \cite{Do12}) and for cubic threefolds (\cite{WY20}).  For ${\rm dim}(X)=4$, recently Laza--Zheng \cite{LZ22} classified the symplectic automorphism groups $\Aut^s(X)$ of cubic fourfolds and proved that the Fermat cubic fourfold has the largest possible order for $|\Aut(X)|$. For some partial results on abelian subgroups of automorphism groups of smooth cubic hypersurfaces of arbitrary dimension, see \cite{GL11}, \cite{Zh22}, \cite{GLM23}.  However, a classification of all possible subgroups of $\Aut(X)$ for cubic fourfolds is still unknown and such classifications for dimensions $\ge 5$ are widely open. Our main results of this paper completely solve this problem for cubic fivefolds and fourfolds.	
	
	\begin{theorem}[Theorem \ref{thm:Main}]
		A finite group $G$ can act faithfully on a smooth cubic fivefold if and only if $G$ is isomorphic to a subgroup of one of the following $20$ groups:
		\begin{table}[H]\rm
		\renewcommand\arraystretch{1.1}
				\begin{tabular}{cccccc}
					\hline
					No. &group&order&	No. &group&order\\
						\hline
					1 & $C_3^6 \rtimes S_7$ &           3674160                                                               & 11 & $C_{63} \rtimes C_6$       &                                       378              \\
					2 & $((C_3^2 \rtimes C_3)\rtimes C_4) \times (C_3^3 \rtimes S_4)$ &69984& 12 &$C_3 . M_{10}$       &                       2160           \\
					3 & $C_8 \times (C_3^3 \rtimes S_3)$    &1296                                                & 13 &  $S_7 \times C_3$                      &    15120       \\
					4 & $S_5 \times (C_3^3 \rtimes S_3)$  &                     19440                             & 14 & $C_3 \times ((C_8 \times C_2)\rtimes C_2)$    &    96   \\
					5 & $C_{48} \times S_3$                                                                           & 288& 15 & $C_3 \times ({\rm PSL}(3,2) \rtimes C_2)$&1008 \\
					6 & $\rm{PSL}(2,11) \times (C_3^2 \rtimes C_2)$                                  &11880& 16 & $C_3.A_7$     &7560 \\
					7 & $((C_3^2 \rtimes C_3)\rtimes C_4)^2\rtimes C_2$                       & 23328&17 &   $ C_3 \times \GL(2,3)$ &  144 \\
					8 & $((C_3^2 \rtimes C_3)\rtimes C_4) \times C_8$                             &864& 18&  $((C_3^2\rtimes  C_3)\rtimes Q_8)\rtimes C_3 $               & 648\\
					9 & $S_5 \times ((C_3^2 \rtimes C_3)\rtimes C_4)$                              & 12960& 19 & $C_{64}$ &   64 \\
					10 & $C_{96}$  &96& 20 & $C_{43} \rtimes C_7$                                      &  301\\           
						\hline
				\end{tabular}

		\end{table}
	\end{theorem}
		
	\begin{theorem}[Theorem \ref{thm:fourfold}]
		A finite group $G$ can act faithfully on a smooth cubic fourfold if and only if $G$ is isomorphic to a subgroup of one of the following $15$ groups:
		\begin{table}[H]\rm
		\renewcommand\arraystretch{1.1}
	\begin{tabular}{cccccc}
		\hline
		No. &group&order&	No. &group&order\\
		\hline
					1 & $C_3^5 \rtimes S_6$                              &174960                                                                & 9 & $C_{21} \rtimes C_6$                      &126                      \\
					2 & $((C_3 \times (C_3^3\rtimes C_3))\rtimes C_3) \rtimes (C_4 \times C_2)$  &5832  & 10 & $M_{10}$                    &720   \\
					3 & $C_8 \times (C_3^2 \rtimes C_2)$          &144                                                             & 11 & $S_7$                                 &5040  \\
					4 & $S_5 \times (C_3^2 \rtimes C_2)$                                            &2160                            & 12 & $(C_8 \times C_2)\rtimes C_2$                                         &32      \\
					5 & $C_{48}$            &48                                                                                                      & 13 &  ${\rm PSL}(3,2) \rtimes C_2$&336  \\
					6 & $\rm{PSL}(2,11) \times C_3$                                     &1980                                             & 14 & $\GL(2,3)$      &48   \\
					7 & $((C_3 \times(C_3^2\rtimes C_3))\rtimes C_3)\rtimes (C_4^2 \rtimes C_2)$&7776& 15 &     $(C_3^2\rtimes Q_8)\rtimes C_3 $                                               &216           \\
					8 & $C_{32}$& 32&& &                                                               \\
					\hline                  
				\end{tabular}

		\end{table}
	\end{theorem}
	
	Explicit examples of cubic fivefolds and fourfolds acted on by these maximal groups are given in the examples in Subsections \ref{mainex} and \ref{ex:4folds} respectively.
	
	Next, we briefly explain the idea of the proof of the main results. Let $(n,d)$ be a pair of integers satisfying $n\geq 2$, $d\geq 3$, and $(n,d)\neq(2,4)$. We say a finite group $G$ is an {\it $(n,d)$-group} if $G$ is isomorphic to a subgroup of the automorphism group of a smooth hypersurface in $\P^{n+1}$ of degree $d$. Matsumura--Monsky \cite{MM63} proved that for a smooth hypersurface $X_{n,d}\subset\mathbb{P}^{n+1}$ of degree $d$, its automorphism group $\Aut(X_{n,d})$ is a finite group, and
	\begin{center}
		$\Aut(X_{n,d})=\{\phi\in{\rm PGL}(n+2,\mathbb{C})~|~\phi(X_{n,d})=X_{n,d}\}.$
	\end{center}
	Two smooth hypersurfaces of dimension $n$ and degree $d$ are isomorphic if and only if they are projectively equivalent, that is, their defining equations are the same up to linear change of coordinates. Therefore, classifying all $(n,d)$-groups is equivalent to classifying all finite subgroups of $\PGL(n+2,\C)$ preserving smooth homogeneous polynomials of degree $d$. By solving the latter problem in the classical invariant theory, Oguiso--Yu \cite{OY19} classified all groups acting faithfully on smooth quintic threefolds, which meanwhile gives a systematic (and computer-aided) method for classifying all possible $(n,d)$-groups for prescribed integers $n$ and $d$. Based on this method, Wei--Yu \cite{WY20} completed the classification of all $(3,3)$-groups. In this paper, we follow the approach of Oguiso--Yu's work to study the automorphism groups of cubic fivefolds. However, the dimensions of target hypersurfaces in this paper are higher and  groups of automorphisms in question are more complex. In order to overcome such difficulties, among other things, we introduce two new notions, {\it partitionability} (Definition \ref{def:parti}) and {\it characteristic sets} (Definition \ref{def:SrF}) of homogeneous polynomials. Partitionability and characteristic sets are crucial for our classification of $(5,3)$-groups since they not only significantly simplify the classification procedures conceptually but also considerably reduce the amount of calculations to rule out relevant groups. In fact, using characteristic sets, we are able to control the automorphism groups of cubic fivefolds $X_F$ defined by cubic polynomials $F$ of the form 
	\begin{equation}\label{F=H+K}
	F=H(x_1,...,x_a)+K(x_{a+1},...,x_{7}), \, 2\le a\le 3
	\end{equation}
	(Propositions \ref{prop:unp+fem}, \ref{prop:1+3+3}, \ref{prop:4+3}), which immediately gives all possible subgroups of $\Aut(X_F)$ from previously known classifications of $(2,3)$-groups and $(3,3)$-groups (Proposition \ref{thm:relation}, Theorems \ref{thm:autex}, \ref{thm:5+2}). On the other hand, for cubic fivefolds with defining polynomials not of the form (\ref{F=H+K}), their automorphism groups are bounded by $90720$ (Proposition \ref{prop:boundorder}) and can be effectively handled by Strategy \ref{strategy} which heavily relies on partitionability. In this way, we complete the classification of $(5,3)$-groups. Note that the proofs of all results in Section \ref{ss:parti} are free from computer algebra, but such results fit well with computer algebra (see Corollary \ref{cor:partitionbyA}, Theorem \ref{thm:control}, Strategy \ref{strategy}).  We believe that the  results on partitionability and characteristic sets are interesting in and of themselves and will be applicable to other problems.
	
	Unlike cubic fivefolds, the automorphism groups of cubic fourfolds have no $F$-lifting in general (see Theorem \ref{thm:F-liftable} and Remark \ref{rem:A7M10}), which is a key obstruction in our classification of such groups. To deal with this issue, we introduce the notion of {\it $C_d$-covering group} (Definition \ref{def:Cdcovering}). A finite group is a $(4,3)$-group only if it has a $C_3$-covering $(5,3)$-group (see Lemma \ref{lem:(4,3)and(5,3)}), which gives strong constraints on $(4,3)$-groups. Moreover, partitionability of defining polynomials of cubic fourfolds and fivefolds are closely related (see Lemma \ref{lem:widehatF}). Based on these relations, we quickly classify all $(4,3)$-groups by taking advantages of our results and strategies for $(5,3)$-groups (see Theorem \ref{thm:ruleout43} and Strategy \ref{strategy43}). As a by-product, we obtain explicit defining polynomials $F_{A_7}$ and $F_{M_{10}}$ of two cubic fourfolds with maximal symplectic automorphism groups $A_7$ and $M_{10}$ respectively (Theorems \ref{thm:A7} and \ref{thm:M10}). To the best of our knowledge, defining equations of the two cubic fourfolds are previously unknown (see \cite[Page 1461]{LZ22}).
	
	We conclude the introduction by explaining in detail which results rely to what extent on computer calculations. All results in Sections \ref{ss:parti} and \ref{ss:examples} are free from computer calculations except that for $14\le i\le 18$, we use computer algebra (Mathematica \cite{Wo}, Magma \cite{BCP}) to verify smoothness of $X_i$ and the inclusions $G_{X_i}\subseteq {\rm Aut}(X_i)$ in Subsection \ref{mainex}. In Sections \ref{sec:5} and \ref{ss:fourfolds}, the results which rely on computer calculations are Theorems \ref{thm:Main}, \ref{thm:control}, \ref{thm:abelian}, \ref{thm:sylow23}, \ref{thm:solvable}, \ref{thm:fourfold}, \ref{thm:ruleout43}, \ref{thm:A7}, \ref{thm:M10}, Remark \ref{rem:Aut=G}, Lemma \ref{lem:ab53}, and Proposition \ref{prop:boundorder}. In fact, we prove Theorems \ref{thm:control}, \ref{thm:abelian}, \ref{thm:ruleout43} using classification of $(5,3)$-representations (see Definition \ref{def:(n,d)-rep}) of relevant abelian groups. Such representations can be computed by hand in principle (see e.g., \cite[Theorem 5.4]{WY20} and Example \ref{ex:C2}), but we use computer algebra (Mathematica) for efficiency. For Remark \ref{rem:Aut=G}, Lemma \ref{lem:ab53}, we use computer algebra (GAP \cite{GAP}) to compute all subgroups of $G_{X_i}$ ($1\le i\le 20$). Note that smoothness of the examples in Subsection \ref{ex:4folds} and the structure description of their automorphism groups follows from Subsetion \ref{mainex} and Remark \ref{rem:Aut=G}. Our proofs of Theorems \ref{thm:Main},  \ref{thm:sylow23}, \ref{thm:solvable}, \ref{thm:M10} heavily rely on computer calculations using Strategy \ref{strategy}. More precisely, we use GAP to do the sub-test (see Remark \ref{rem:subtest}) in the \textbf{Step 1} of Strategy \ref{strategy}; we use a mixture of GAP, Mathematica and Sage \cite{Sage} to compute special almost $(5,3)$-representations (see Definitions \ref{def:(n,d)-rep} and \ref{def:special}) and invariant cubic forms in the \textbf{Steps 2} and \textbf{3} of Strategy \ref{strategy}. Similarly, our proof of Theorem \ref{thm:fourfold} heavily relies on computer calculations using Strategy \ref{strategy43}. Moreover, Proposition \ref{prop:boundorder} (resp. Theorem \ref{thm:A7}) is free from computer calculations modulo Theorems \ref{thm:control} and \ref{thm:sylow23} (resp. Theorem \ref{thm:Main} and Remark \ref{rem:Aut=G}). The computer codes and outputs needed in Sections \ref{ss:examples}-\ref{ss:fourfolds} are contained in the ancillary files to \cite{YYZ23} whose roles are described in Appendix \ref{Appendix}. These files are explicitly mentioned in the relevant proofs and can be obtained at \href{https://arxiv.org/src/2308.07186/anc}{https://arxiv.org/src/2308.07186/anc}.
		
		\subsection*{Acknowledgement} We would like to thank Professor Jun-Muk Hwang for pointing out the notion of Thom--Sebastiani polynomials in Remark \ref{rem:TS} and Professor Keiji Oguiso for helpful discussions on automorphisms of hyperk\"{a}hler manifolds of $K3^{[2]}$-type. 
		 We would also like to express our thanks to the editors and referees for their valuable comments and suggestions.
		This work is partially supported by the National Natural Science Foundation of China (No. 12171351, No. 12071337, No. 11831013,  No. 11921001).

	\section{Notation}
	
	(2.1)  Let $F=F(x_1,...,x_n)$ be a homogeneous polynomial of degree $d$. For $A=(a_{ij})\in \GL(n,\mathbb{C})$, we denote by $A(F)$ the homogeneous polynomial
 $$
  F(\sum_{i=1}^{n}a_{1i}x_i,\cdots,\sum_{i=1}^{n}a_{ni}x_i).
 $$
Then we have $A(F)(x_1,\dots,x_n)=F((x_1,\dots,x_n)A^T)$.
 Note that $(AB)(F)=B(A(F))$ for any $A,B\in \GL(n,\C).$ We define $$G_F:=\{ A\in {\rm GL}(n,{\mathbb C})~|~A(F)=F \}.$$

\medskip

 (2.2) We denote by $X_F\subset \mathbb{P}^{n-1}$ the  hypersurface defined by $F$. If $X_F$ is smooth, we say $F$ is a smooth form of degree $d$. We define $\widehat{F}:=F+x_{n+1}^d$. We denote by $X_i$ (resp. $G_{X_i}$), $i\in\{1, 2,\dots,20\}$, the $20$ smooth cubic fivefolds (resp. finite groups) in the examples in Subsection \ref{mainex}. We denote by $X_i'$, $i\in\{1, 2, \dots,15\}$, the $15$ smooth cubic fourfolds in the examples in Subsection \ref{ex:4folds}. We use $\pi: \GL(n,\mathbb{C})\rightarrow \PGL(n,\mathbb{C})$ to denote the natural quotient map, and for $A\in \GL(n,\mathbb{C})$, we denote $\pi(A)$ by $[A]$. 

\medskip

       (2.3) We say a finite group $G$ is an $(n,d)$-group if $G$ is isomorphic to a subgroup of the automorphism group ${\rm Aut}(X)$ of a smooth hypersurface $X\subset \P^{n+1}$ of degree $d$. Let $p$ be a prime number. If no confusion causes, we use $G_p$ to denote a Sylow $p$-subgroup of $G$. We use $N\rtimes H$ to denote one of the semidirect products of groups $N$, $H$. We use $N.H$ to denote a finite group which fits in a (non-split) short exact sequence of finite groups \begin{equation*}
  \xymatrix@C=0.5cm{ 
  1 \ar[r] & N 
  \ar[r]^{} & N.H
  \ar[r]^{} & H 
  \ar[r] & 1.} 
\end{equation*} Some symbols frequently used in this paper are as follows:
\begin{flushleft}
\begin{tabular}{rl} 
 $\xi_k$ &  the $k$-th primitive root $e^{\frac{2\pi i}{k}}$ of unity, where $k$ is a positive integer; \\
 
 $I_n$ & the identity matrix of rank $n$; \\
 
 $C_n$ & the cyclic group of order $n$;\\
 
 $D_{2n}$ & the dihedral group of order $2n$;\\
 
 $S_n$ & the symmetric group of degree $n$;\\
 
 $A_n$ & the alternating group of degree $n$;\\
 
 $Q_n$ & the quaternion group of order $n$;\\
 
 $\PSL(n,q)$ & the projective special linear group of degree $n$ over the field $\mathbb{F}_q$ with $q$ elements;\\
 
  $\GL(2,3)$ & the general linear group of degree $2$ over the field $\mathbb{F}_3$;\\
 
   $M_{10}$ & the Mathieu group of order $720$;\\
   
   $QD_{16}$ & the quasidihedral group of order $16$;\\

   ${\rm PSU}(3,3)$ & the projective special unitary group of degree $3$ for the quadratic extension field $\mathbb{F}_{9}$\\
   & over the field $\mathbb{F}_3$.\\
 
\end{tabular} 
\end{flushleft}

	\section{Partitionability and characteristic sets}\label{ss:parti}
	In this section, we introduce two notions, partitionability and characteristic sets. We use characteristic sets to control symmetries of homogeneous polynomials $F$ having partitions of certain types (Propositions \ref{prop:unp+fem}, \ref{prop:1+3+3}, \ref{prop:4+3}) and we give a relation among partitionability, $F$-liftability, and automorphism groups of hypersurfaces (Proposition \ref{thm:relation}). These results are very important for our classification of $(5,3)$-groups in later sections.
	
	\begin{definition}\label{def:parti}
		Let $F=F(x_1,x_2,\dots,x_{m})$ be a homogeneous polynomial of degree $d$. If there exists an invertible matrix $A\in{\rm GL}(m,\mathbb{C})$ and positive integers $a_1,\dots,a_t$ such that 
		$$A(F)=H_1(x_1,\dots,x_{a_1})+H_2(x_{a_1+1},\dots,x_{a_1+a_2})+\cdots+ H_t(x_{a_1+a_2+\cdots+a_{t-1}+1},\dots,x_{a_1+a_2+\cdots+a_t}),$$
		where $a_1+\cdots+a_t\leq m$ and $t\geq2$, then we say $F$ is {\it partitionable} or $F$ has an $(a_1,a_2,\dots,a_t)$-{\it type partition}. In this case, we say $F$ can be partitioned as $H_1+\cdots+H_t$. Otherwise, we say $F$ is {\it unpartitionable}. 
		
		If $F$ has an $(a_1,a_2,\dots,a_t)$-{\it type partition} given by $A(F)=H_1+\cdots+H_t$ and all $H_i$ ($i=1,...,t$) are unpartitionable, we say $F$ has a {\it maximal} $(a_1,a_2,\dots,a_t)$-type partition. 
	\end{definition}
	
	\begin{remark}\label{rem:TS}
	In the literature, polynomials of the form $H_1(x_1,...,x_{a_1})+H_2(x_{a_1 +1},...,x_{a_1 +a_2})$ are also called Thom--Sebastiani (type) polynomials. If $H_i: (\mathbb{C}^{a_i},0)\rightarrow (\mathbb{C},0)$ are two germs of holomorphic functions with isolated critical points, then Thom--Sebastiani theorem (\cite{ST71}) asserts that the vanishing cycles complex of $H_1+H_2$ is isomorphic to the tensor product of those of $H_1$ and $H_2$.
	\end{remark}
	
	\begin{example}
	Let $F=x_1^3+x_2^2x_3+x_3^2 x_2$. Then $F$ has a $(1,2)$-type partition given by $F=H_1+H_2$, where $H_1=x_1^3$ and $H_2=x_2^2x_3+x_3^2 x_2$. This partition is not maximal. In fact, $H_2$ has a $(1,1)$-type partition since $A(H_2)=x_2^3+x_3^3$, where $A=\begin{pmatrix}-1&-1\\ \frac{1-\sqrt{3} i}{2}& \frac{1+\sqrt{3} i}{2}\end{pmatrix}$. 
	\end{example}
	
	More generally, for partitions of cubic forms, we have the following result. 
	\begin{lemma}\label{lem:2n3n}
		Consider a smooth cubic form $F=F(x_1,\dots,x_{m})$ with $m\geq4$. If $F$ has a $(2,m-2)$-type partition, then $F$ has a $(1,1,m-2)$-type partition. If $F$ has a $(3,m-3)$-type partition, then either $F$ has a $(1,1,1,m-3)$-type partition or there exists $A\in\GL(m,\C)$ such that  $A(F)=x_1^3+x_2^3+x_3^3+\lambda x_1x_2x_3+H(x_4,\dots,x_m)$, where $\lambda\neq0$ $($see \cite[Section 3.1.2]{Do12}$)$.
	\end{lemma}
	
		Let $F(x_1,\dots,x_m)$ be a homogeneous polynomial of degree $d$. For $1\leq i\leq d$, we define the natural $i$-th order differential mapping induced by $F$ as follows:
	$$D_i^F:\mathcal{D}_i(x_1,\dots,x_m)\longrightarrow\C[x_1,\dots,x_m],$$
	where $\mathcal{D}_i(x_1,\dots,x_m)$ denotes the vector space of $i$-th differential operators. 
	For example, $D_1^F(\frac{\partial }{\partial x_i})=\frac{\partial F}{\partial x_i}$, $D_2^F(\frac{\partial^2 }{\partial x_i \partial x_j})=\frac{\partial^2 F}{\partial x_i \partial x_j}$. Let $\mathfrak{m}=\mathfrak{m}(x_1,\dots,x_{m})$ be a monomial of degree $d$. 
	Then we say $\mathfrak{m}$ is in $F$ (or $\mathfrak{m}\in F$) if the coefficient of $\mathfrak{m}$ is not zero in the expression of $F$. Oguiso--Yu \cite{OY19} introduced the differential method to classify the automorphism groups of smooth quintic threefolds. In order to study partitionability of polynomials and their symmetries, we recall the differential method.

	\begin{thm}[{\cite[Theorem 3.5]{OY19}}]\label{thm:dm}
		Let $F(x_1,\dots,x_m)$, $G(y_1,\dots,y_m)$ be nonzero homogeneous polynomials of degree $d$. Suppose that there exists an invertible matrix $A=(a_{ij})_{1\leq i,j\leq m}$, such that $F(x_1,\dots,x_m)=$ $G(\sum_{i=1}^{m}a_{1i}x_i,\dots,\sum_{i=1}^{m}a_{mi}x_i)$, then ${\rm rk}(D_i^F)={\rm rk}(D_i^G)$, for all $1\leq i\leq d$.
	\end{thm}
	
	We now introduce the new notion of characteristic sets and prove that they are invariants of homogeneous polynomials up to linear transformations.
	\begin{definition}\label{def:SrF}
	 Let $F(x_1,\dots,x_m)$ be a homogeneous polynomial of degree $d\geq 2$.  Let $r$ be a positive integer. We define 
	$$S_r^F:=\{(l_1,\dots,l_m)\in {\mathbb C}^m\mid {\rm rk} (D_1^{l_1\frac{\partial F}{\partial x_1}+l_2\frac{\partial F}{\partial x_2}+\cdots+l_m\frac{\partial F}{\partial x_{m}}})=r\}.$$
	We call $S_r^F$ the {\it $r$-th characteristic set} of $F$. We define $V_r^F\subseteq {\mathbb C}^m$ to be the subspace spanned by $S_r^F$.
	\end{definition}
	 Then we have the following lemma.

	\begin{lemma}\label{lem:inv}
		Let $G=G(y_1,y_2,\dots,y_m)$ be a homogeneous polynomial of degree $d$, and let $F(x_1,\dots,x_m)=G(\sum_{i=1}^{m}a_{1i}x_i,\dots,\sum_{i=1}^{m}a_{mi}x_i)$, where $A=(a_{ij})_{1\leq i,j\leq m}\in {\rm GL}(m,{\mathbb C})$. 
		Consider the following linear transformation:
		$$P:{\mathbb C}^m\longrightarrow {\mathbb C}^m, \; (l_1,\dots,l_m)\longmapsto (\sum_{j=1}^{m}a_{1j}l_j,\dots,\sum_{j=1}^{m}a_{mj}l_j).$$
		Then $P(S_r^F)=S_r^G$ and $P(V_r^F)=V_r^G$. 
		In particular, if $A(F)=F$, then $P(S_r^F)=S_r^F$ and $P(V_r^F)=V_r^F$. 
	\end{lemma}
	
	\begin{proof}
		By the proof of \cite[Theorem 3.5]{OY19}, we have the following commutative diagram:
		\[\xymatrixcolsep{5pc}\xymatrix{
			\mathcal{D}_1(x_1,\dots,x_m) \ar[d]^p \ar[r]^{D_1^F} &\mathbb{C}[x_1,\dots,x_m]\\
			\mathcal{D}_1(y_1,\dots,y_m) \ar[r]^{D_1^G} &\mathbb{C}[y_1,\dots,y_m]\,\, ,\ar[u]^q}\]
		where the isomorphisms $p$ and $q$ are defined as follows:
		$$p(\frac{\partial}{\partial x_i})=\sum_{j=1}^{m}{a_{ji}\frac{\partial}{\partial y_j}},~~~~~~q(H(y_1,\dots,y_m))=H(\sum_{i=1}^{m}{a_{1i}x_i},\dots,\sum_{i=1}^{m}{a_{mi}x_i}).$$ 
		Then we have 
		$$p({\frac{\partial }{\partial x_1}l_1+\frac{\partial }{\partial x_2}l_2+\cdots+\frac{\partial}{\partial x_{m}}l_{m}})=\sum_{j=1}^{m}{a_{j1}l_1\frac{\partial}{\partial y_j}}+\sum_{j=1}^{m}{a_{j2}l_2\frac{\partial}{\partial y_j}}+\cdots+\sum_{j=1}^{m}{a_{jm}l_m\frac{\partial}{\partial y_j}}.$$
		Note that the coefficients of $\frac{\partial}{\partial y_j}$ on the right-hand side correspond to the images of $(l_1,\dots,l_m)$ under the mapping $P$. 
		
		Since $q$ induces a linear change of coordinates from the polynomial $\sum_{j=1}^{m}{a_{j1}l_1\frac{\partial G}{\partial y_j}}+\cdots+\sum_{j=1}^{m}{a_{jm}l_m\frac{\partial G}{\partial y_j}}$ to the polynomial $\frac{\partial F}{\partial x_1}l_1+\cdots+\frac{\partial F}{\partial x_{m}}l_{m}$, by Theorem \ref{thm:dm}, we have
		$${\rm rk}(D_1^{\sum_{j=1}^{m}{a_{j1}l_1\frac{\partial G}{\partial y_j}}+\cdots+\sum_{j=1}^{m}{a_{jm}l_m\frac{\partial G}{\partial y_j}}})={\rm rk}(D_1^{\frac{\partial F}{\partial x_1}l_1+\cdots+\frac{\partial F}{\partial x_{m}}l_{m}}).$$
		Since $A$ is an invertible matrix, we have $P(S_r^F)=S_r^G$, and hence $P(V_r^F)=V_r^G$.
		\end{proof}
	
	\begin{remark} 
			Note that by the mapping $(l_1,\dots,l_m)\longmapsto (l_1:\cdots:l_m)$, we can view $S_r^F$ as a subset of ${\mathbb P}^{m-1}$ and $V_r^F$ as a linear subspace of ${\mathbb P}^{m-1}$.
		\end{remark}
		
		Next, we introduce the relationship between  characteristic sets and partitionability of cubic forms.

	\begin{lemma} \label{lem:linear}
		Let $F=F(x_1,x_2,\dots,x_{m})$ be a smooth cubic form. Then $F$ has a $(1,m-1)$-type partition if and only if the first characteristic set $S_1^F\neq \emptyset$.
	\end{lemma}
	
	\begin{proof}
		The ``only if" part is clear. To prove ``if" part, we assume $(l_1,\dots,l_{m})\in S_1^F$.
		Since it is clear that not all of $l_1,\dots,l_{m}$ are zero, there exists an invertible matrix $B=(b_{ij})_{1\leq i,j\leq m}$ such that $b_{i1}=l_i$, for $1\leq i\leq m$. Consider the linear change of coordinates:$$(x_1,...,x_{m})^T=B\cdot (z_1,...,z_{m})^T.$$
		Here $T$ denotes the transpose. By substituting, we define $$G(z_1,...,z_{m})=F(\sum_{i=1}^{m}b_{1i}z_i,\dots,\sum_{i=1}^{m}b_{mi}z_i).$$
		 Applying $\frac{\partial}{\partial z_1}$ to both sides of the equality, we have $$\frac{\partial G}{\partial z_1}=\frac{\partial F}{\partial x_1}\frac{\partial x_1}{\partial z_1}+\cdots+\frac{\partial F}{\partial x_{m}}\frac{\partial x_{m}}{\partial z_1}=\frac{\partial F}{\partial x_1}l_1+\cdots+\frac{\partial F}{\partial x_{m}}l_{m}.$$ 
		By Theorem \ref{thm:dm}, 
		$${\rm rk}(D_1^{\frac{\partial G}{\partial z_1}})={\rm rk}(D_1^{\frac{\partial F}{\partial x_1}l_1+\cdots+\frac{\partial F}{\partial x_{m}}l_{m}})=1.$$
		Since $\frac{\partial G}{\partial z_1}$ is a homogeneous quadratic polynomial in variables $z_1,\dots,z_m$, we infer that $\frac{\partial G}{\partial z_1}=(\lambda_1z_1+\lambda_2z_2+\cdots+\lambda_{m}z_{m})^2$, where $\lambda_1,\dots,\lambda_{m} \in \C$. 
		
		If $\lambda_1=0$, then $G=z_1(\lambda_2z_2+\dots+\lambda_{m}z_{m})^2+H(z_2,...,z_m)$, where $H$ is a homogeneous polynomial of degree $3$. Then $X_G\subset \mathbb{P}^{m-1}$ is singular at $(z_1:z_2:\cdots:z_m)=(1:0:\cdots:0)$, a contradiction. Therefore, we have $\lambda_1\neq 0$ and $G=\frac{1}{3\lambda_1}(\lambda_1z_1+\lambda_2z_2+\cdots+\lambda_{m}z_{m})^3+K(z_2,\dots,z_{m})$, where $K(z_2,\dots,z_{m})$ is a homogeneous polynomial of degree $3$. This implies that $F$ has a $(1,m-1)$-type partition.
	\end{proof}
	
	\begin{remark}\label{rem:m345} 
		Note that for a smooth cubic form $F=F(x_1,...,x_m)$ with $m\in\{3,4,5\}$, $F$ being unpartitionable is equivalent to the first characteristic set $S_1^F=\emptyset$ by Lemmas \ref{lem:2n3n} and \ref{lem:linear}. 
	\end{remark}
	
	Next, we will control shape of matrices that preserve partitionable polynomials of degree $3$. Based on this, we will determine the structure of automorphism groups of cubic hypersurfaces defined by polynomials admitting certain types of partition, which will be crucial for our classification of $(5,3)$-groups and $(4,3)$-groups in later sections.
	
	\begin{proposition}\label{prop:unp+fem} 
		Let $F=F(x_1,\dots,x_{m})$ be a smooth cubic form with $m\ge 4$. Suppose $F=H(x_1,\dots,x_k)+K(x_{k+1},\dots,x_{m})$, where $3\leq k \leq m-1$, $S_1^H=\emptyset$, and $K=x_{k+1}^3+...+x_m^3$. Then $$G_F \subset\{\begin{pmatrix}B&0\\0&C\end{pmatrix}\mid B\in \GL(k,\C), C\in \GL(m-k,\C) \}.$$ 
	\end{proposition}
	
	\begin{proof}
		Note that smoothness of $F$ implies that $H\neq 0$.
		Let $A=(a_{ij})_{1\leq i,j\leq m}\in G_F$ and $y_i=\sum_{j=1}^{m}{a_{ij}x_j},1\leq i\leq m$. By $F=A(F)$, we have $H(x_1,...,x_k)+K(x_{k+1},...,x_m)=H(y_1,...,y_k)+K(y_{k+1},...,y_m)$. Applying $\frac{\partial}{\partial x_t}$ to both sides of this equality for $k+1\leq t\leq m$, we have	
		$$3x_t^2=\sum_{i=1}^{k}\frac{\partial H(y_1,\dots,y_k)}{\partial y_i}{a_{it}}+\sum_{j=k+1}^{m}{3y_j^2a_{jt}}.$$
		Denote the right-hand side as $J$. Then ${\rm rk}(D_1^{J})={\rm rk} (D_1^{3x_t^2})=1$. Since $S_1^H=\emptyset$, we have $a_{it}=0$ for all $1\leq i\leq k$. In fact, suppose for given $t$ there exists $1\leq i\leq k$ with $a_{it}\neq 0$, then from $S_1^H=\emptyset$,  we know the polynomial 
		$J_1=\sum a_{it}\frac{\partial H}{\partial y_i}$
			cannot satisfy $\Rank(D_1^{J_1})=1$. So $\Rank(D_1^{J_1})\geq2$, which implies $\Rank(D_1^J)\geq2$,
			contradiction. Similarly, for each $t\in\{k+1,...,m\}$,  there exists exactly one index $j_t\in\{k+1,...,m\}$ such that $a_{j_t t}\neq 0$. Since $A$ is invertible, without loss of generality, we may assume that $j_t=t$ for all $t\in\{k+1,...,m\}$. Therefore, $H(x_1,...,x_k)+x_{k+1}^3+...+x_m^3$ is equal to $$H(\sum_{i=1}^{k}{a_{1i}x_i},\dots,\sum_{i=1}^{k}{a_{ki}x_i})+K(a_{(k+1)(k+1)}x_{k+1}+\sum_{j=1}^{k}{a_{(k+1)j}x_j},\dots,a_{m m}x_{m}+\sum_{j=1}^{k}{a_{mj}x_j}).$$
		Direct calculation shows that $a_{ij}=0$ for $k+1\leq i\leq m$ and $1\leq j\leq k$. Thus, $A$ is of the desired shape.
	\end{proof}

	\begin{lemma}\label{lem:3vari}
		Let $F(x_1,x_2,x_3)$ be a smooth cubic form. Then $G_F$ is isomorphic to $C_3^2\rtimes S_3$, $C_3^3\rtimes S_3$, or $(C_3^2\rtimes C_3)\rtimes C_4$.
	\end{lemma}
	
	\begin{proof}
	 Smoothness of $F$ implies that for any prime $p>3$, $p\nmid |G_F|$ (see the proof of \cite[Theorem 1.3]{GL13}). Therefore, $|G_F|=2^a\cdot3^b$, where $a,b\geq0$. By some calculations and the smoothness of $F$, we can conclude that
\begin{enumerate}
		\item[(i)] if $|(G_F)_2|\geq 4$, then $B_1(F)= x_1^2x_2+x_2^2x_3+x_3^3$ for some $B_1\in {\rm GL}(3,{\mathbb C})$ and $|(G_F)_2|=4$; and

		\item[(ii)]  if $|(G_F)_3|\geq 3^4$, then $B_2(F)= x_1^3+x_2^3+x_3^3$ for some $B_2\in {\rm GL}(3,{\mathbb C})$ and $|(G_F)_3|=3^4$.
\end{enumerate}
		In fact, if $|(G_F)_2|\geq 4$ (resp. $|(G_F)_3|\geq 3^4$), then $G_F$ contains an abelian subgroup $N$ of order $4$ (resp. $3^3$) and up to conjugation in $\GL(3,\C)$, we have $N=\langle{\rm diag}(\xi_4,-1,1)\rangle\cong C_4$ (resp. $N=\langle {\rm diag}(\xi_3,1,1), {\rm diag}(1,\xi_3,1), {\rm diag}(1,1,\xi_3) \rangle \cong C_3^3$) by smoothness of $F$ (see e.g. \cite[Theorem 7.7]{OY19} and \cite[Theorem 5.2]{WY20}). From this, we infer that (i) and (ii) hold.
		
		On the other hand, by Lemma \ref{lem:2n3n}, we can assume that $F=H_{\lambda}:=x_1^3+x_2^3+x_3^3+\lambda x_1x_2x_3$, where $\lambda\in{\mathbb C}$. Therefore, $C_3^2\rtimes S_3$ is a subgroup of $G_F$ and $3^3\cdot2$ divides $|G_F|$. Moreover, the following statements hold:
		
		(1) If $\lambda=0$, $G_F=C_3^3\rtimes S_3$;
		
		(2) If $\lambda=3(\sqrt{3}-1)$, $G_F$ is
		$$\left\langle \begin{pmatrix}0&1&0\\0&0&1\\1&0&0\\ \end{pmatrix}, \begin{pmatrix}1&0&0\\0&\xi_3&0\\0&0&\xi_3^2\\ \end{pmatrix}, \frac{1}{\sqrt{3}}\begin{pmatrix}1&1&1\\1&\xi_3&\xi_3^2\\1&\xi_3^2&\xi_3\\ \end{pmatrix}\right\rangle \cong(C_3^2\rtimes C_3)\rtimes C_4;$$
		
		(3) If there exists no $B\in {\rm GL}(3,{\mathbb C})$ with $B(F)=H_0, H_{3(\sqrt{3}-1)}$, then $G_F=C_3^2\rtimes S_3$.	\end{proof}

	\begin{proposition}\label{prop:1+3+3}
		Let $F=F(x_1,\dots,x_{7})$ be a smooth cubic form with $F=x_1^3+H(x_2,x_3,x_4)+K(x_5,x_6,x_7)$, where $H$ and $K$ are unpartitionable. If $A\in {\rm GL}(7,{\mathbb C})$ satisfies $A(F)=F$, then either $A=\begin{pmatrix}1&0&0\\0&B&0\\0&0&C\end{pmatrix}$ or $A=\begin{pmatrix}1&0&0\\0&0&B\\0&C&0\end{pmatrix}$, where $B,C\in{\rm GL}(3,\C)$.
	\end{proposition}
	
	\begin{proof}
		Suppose $A=(a_{ij})_{1\leq i,j\leq 7}$. Let $(y_1,...,y_{7})^T=A\cdot (x_1,...,x_{7})^T$. Since $A(F)=F$, we have $F(x_1,\dots,x_7)=F(y_1,\dots,y_7)$. Applying $\frac{\partial }{\partial x_1}$ to both sides of the equality, we get 
		$$3x_1^2=3a_{11}y_1^2+\frac{\partial H(y_2,y_3,y_4)}{\partial x_1}+\frac{\partial K(y_5,y_6,y_7)}{\partial x_1}.$$
		Since $H$ and $K$ are unpartitionable, by Lemma \ref{lem:linear}, $S_1^H$ and $S_1^K$ are empty and thus $S_1^{H+K}$ is also empty. So we have $a_{i1}=0$ for $2\leq i\leq 7$. Then $a_{11}\neq 0$. Since 
		$$F=(a_{11}x_1+a_{12}x_2+a_{13}x_3+a_{14}x_4+a_{15}x_5+a_{16}x_6+a_{17}x_7)^3+A(H+K),$$
		the term $3a_{11}^2a_{1i}\cdot x_1^2x_i$, $2\leq i \leq7$ is in $F$ unless $a_{1i}=0$. So $A=\begin{pmatrix}1&0\\0&A_1\end{pmatrix}$, where $A_1=(a_{ij})_{2\leq i,j\leq 7}$.
	Similarly, for $2\leq j\leq 7$, we have
		$$\frac{\partial H(x_2,x_3,x_4)}{\partial x_j}+\frac{\partial K(x_5,x_6,x_7)}{\partial x_j}=\frac{\partial H(y_2,y_3,y_4)}{\partial x_j}+\frac{\partial K(y_5,y_6,y_7)}{\partial x_j}.$$
		By $S_1^H=S_1^K=\emptyset$, we get $${\rm rk}(D_1^{\frac{\partial H(x_2,x_3,x_4)}{\partial x_j}+\frac{\partial K(x_5,x_6,x_7)}{\partial x_j}})\in \{2,3\}$$
		and $${\rm rk}(D_1^{\frac{\partial H(y_2,y_3,y_4)}{\partial x_j}}), {\rm rk}(D_1^{\frac{\partial K(y_5,y_6,y_7)}{\partial x_j}})\in \{0,2,3\}.$$
		Then either $a_{2j}=a_{3j}=a_{4j}=0$ or $a_{5j}=a_{6j}=a_{7j}=0$.
		
		We may assume that $H(x_2,x_3,x_4)=x_2^3+x_3^3+x_4^3+\lambda_Hx_2x_3x_4$ and  $K(x_5,x_6,x_7)=x_5^3+x_6^3+x_7^3+\lambda_Kx_5x_6x_7$, where $\lambda_H\lambda_K\neq 0$. If $a_{22}=a_{32}=a_{42}=0$, we can infer that $a_{23}=a_{33}=a_{43}=0$ and $a_{24}=a_{34}=a_{44}=0$ (since otherwise the monomial $x_2x_3x_4$ is in neither $H(y_2,y_3,y_4)$ nor $K(y_5,y_6,y_7)$, which is a contradiction). Since A is invertible, we have $A=\begin{pmatrix}1&0&0\\0&0&B\\0&C&0\end{pmatrix}$. For the case $a_{52}=a_{62}=a_{72}=0$, we have $A=\begin{pmatrix}1&0&0\\0&B&0\\0&0&C\end{pmatrix}$ by similar arguments.
	\end{proof}
	
	Next, we consider the case where $F(x_1,\dots,x_{7})$ can be partitioned into $H(x_1,\dots,x_3)+K(x_4,\dots,x_{7})$, where $H$ and $K$ are both unpartitionable. For this goal, we need the following result.
	
	\begin{lemma}\label{lem:hyp}
		Let $F=F(x_1,x_2,x_3,x_4)$ be an unpartitionable smooth cubic form. Then the subspace $V\subseteq \C^4$ generated by the union $S_1^F\cup S_2^F\cup S_3^F$ is of dimension at least $3$. In particular, there exists an invertible matrix $A=(a_{ij})_{1\leq i,j\leq 4}$ such that $K(x_1,x_2,x_3,x_4):=A(F)$ satisfies the following:
		$$\Rank (D_1^\frac{\partial K}{\partial x_i})\leq 3,~~ i=1,2,3.$$
	\end{lemma}
	
	\begin{proof}
		Let $l_1,l_2,l_3,l_4\in\C$ and at least one of them is not zero. Considering the polynomial $N:=\frac{\partial F}{\partial x_1}l_1+\cdots+\frac{\partial F}{\partial x_4}l_4$, we have:
		
		$$\frac{\partial N}{\partial x_1}=b_{11}(l_1,\dots,l_4)x_1+\cdots+b_{14}(l_1,\dots,l_4)x_4,$$
		$$\vdots$$
		$$\frac{\partial N}{\partial x_4}=b_{41}(l_1,\dots,l_4)x_1+\cdots+b_{44}(l_1,\dots,l_4)x_4,$$
		where $b_{ij}$ are linear forms of $l_1,\dots,l_4$.
		
		Note that ${\rm rk}(D_1^N)={\rm rk}\begin{pmatrix}b_{11}&\cdots&b_{14}\\ \vdots&\ddots&\vdots\\b_{41}&\cdots&b_{44}\end{pmatrix}$. Then det$\begin{pmatrix}b_{11}&\cdots&b_{14}\\ \vdots&\ddots&\vdots\\b_{41}&\cdots&b_{44}\end{pmatrix}=0$ defines a hypersurface in $\P^3$. From this, we deduce the lemma.
	\end{proof}
	
	\begin{proposition} \label{prop:4+3}
		Let $F=F(x_1,\dots,x_{7})$ be a smooth cubic form such that $F(x_1,\dots,x_{7})=H(x_1,\dots,x_3)+K(x_4,\dots,x_{7})$, where $H$ and $K$ are unpartitionable. Then $$G_F \subset\{\begin{pmatrix}B&0\\0&C\end{pmatrix}\mid B\in \GL(3,\C), C\in \GL(4,\C) \}.$$
	\end{proposition}
	\begin{proof}
		By assumption, $S_1^H=S_1^K=\emptyset$ (see Remark \ref{rem:m345}). As above, we may assume $H(x_1,x_2,x_3)=x_1^3+x_2^3+x_3^3+\lambda x_1x_2x_3$, $\lambda\neq0$.
		Then we have ${\rm rk}(D_1^\frac{\partial H}{\partial x_i})=3$, $i=1,2,3$.
		
		By Lemma \ref{lem:hyp}, we may assume that ${\rm rk}(D_1^\frac{\partial K}{\partial x_i})\leq 3$, $i=4,5,6$.
		
		Suppose $A=(a_{ij})_{1\leq i,j\leq 7}\in G_F$. Since $F(x_1,\dots,x_7)=F(\sum_{i=1}^{7}a_{1i}x_i,\dots,\sum_{i=1}^{7}a_{7i}x_i)$, we have 
		\begin{center}
			${\rm rk}(D_1^\frac{\partial F}{\partial x_i})={\rm rk}(D_1^{\frac{\partial F}{\partial x_1}a_{1i}+\cdots+\frac{\partial F}{\partial x_7}a_{7i}})$, $\forall i=1,\dots,7$.
		\end{center}
		Then
		\begin{equation*}
			\begin{aligned}
				3={\rm rk}(D_1^\frac{\partial F}{\partial x_1})
				&={\rm rk}(D_1^{\frac{\partial H}{\partial x_1}a_{11}+\cdots+\frac{\partial H}{\partial x_3}a_{31}+\frac{\partial K}{\partial x_4}a_{41}+\cdots+\frac{\partial K}{\partial x_7}a_{71}})\\
				&={\rm rk}(D_1^{\frac{\partial H}{\partial x_1}a_{11}+\cdots+\frac{\partial H}{\partial x_3}a_{31}})+{\rm rk}(D_1^{\frac{\partial K}{\partial x_4}a_{41}+\cdots+\frac{\partial K}{\partial x_7}a_{71}}).
			\end{aligned}
		\end{equation*}
		Thus, by $S_1^H = S_1^K =\emptyset$ and smoothness	of $H, K$, we have either (1) $a_{11}=a_{21}=a_{31}=0$ or (2) $a_{41}=a_{51}=a_{61}=a_{71}=0$.
		
		For Case (1), by similar arguments, we have either $a_{12}=a_{22}=a_{32}=0$ or $a_{42}=a_{52}=a_{62}=a_{72}=0$. We define $$\tilde{H}:=H(\sum_{i=1}^{7}a_{1i}x_i,\sum_{i=1}^{7}a_{2i}x_i,\sum_{i=1}^{7}a_{3i}x_i),$$ and $$\tilde{K}:=K(\sum_{i=1}^{7}a_{4i}x_i,\sum_{i=1}^{7}a_{5i}x_i,\sum_{i=1}^{7}a_{6i}x_i,\sum_{i=1}^{7}a_{7i}x_i).$$
		Since $\tilde{H}+\tilde{K}=H(x_1,x_2,x_3)+K(x_4,x_5,x_6,x_7)$ and $a_{11}=a_{21}=a_{31}=0$, it follows that $\tilde{H}$ has no monomial involving $x_1$. If $a_{42}=a_{52}=a_{62}=a_{72}=0$, then there is no monomial involving $x_2$ in $\tilde{K}$, which contradicts $\lambda x_1x_2x_3\in\tilde{H}+\tilde{K}$. Thus $a_{12}=a_{22}=a_{32}=0$. Similarly, $a_{13}=a_{23}=a_{33}=0$. 
		
		Since $A$ is invertible, and ${\rm rk}(D_1^\frac{\partial F}{\partial x_i})\leq 3$, $i=4,5,6$, we infer that $a_{4i}=a_{5i}=a_{6i}=a_{7i}=0$, $i=4,5,6$. Again, by $\tilde{H}+\tilde{K}=H+K$, we have $$K(x_4,\dots,x_7)=(a_{14}x_4+\cdots+a_{17}x_7)^3+(a_{24}x_4+\cdots+a_{27}x_7)^3+(a_{34}x_4+\cdots+a_{37}x_7)^3$$$$+\lambda (a_{14}x_4+\cdots+a_{17}x_7)(a_{24}x_4+\cdots+a_{27}x_7)(a_{34}x_4+\cdots+a_{37}x_7)+K(a_{47}x_7,a_{57}x_7,a_{67}x_7,a_{77}x_7).$$
		It is clear that the last term on the right hand side of the equation is equal to $\alpha x_7^3$, $\alpha \in\C$. Then under the linear change of coordinates:
		$$(\tilde{x_4},\tilde{x_5},\tilde{x_6},\tilde{x_7})^T=\begin{pmatrix}a_{14}&a_{15}&a_{16}&a_{17}\\a_{24}&a_{25}&a_{26}&a_{27}\\a_{34}&a_{35}&a_{36}&a_{37}\\0&0&0&1\\\end{pmatrix}\cdot(x_4,x_5,x_6,x_7)^T,$$
		we have $K(x_4,x_5,x_6,x_7)=\tilde{x_4}^3+\tilde{x_5}^3+\tilde{x_6}^3+\lambda\tilde{x_4}\tilde{x_5}\tilde{x_6}+\alpha \tilde{x_7}^3,$ which leads to a contradiction to the condition $K$ is unpartitionable. Therefore, Case (1): $a_{11}=a_{21}=a_{31}=0$ is impossible.
		
		For Case (2), by similar arguments as in Case (1), we can deduce that  $a_{4i}=a_{5i}=a_{6i}=a_{7i}=0$ for $i=1,2,3$ and $a_{1j}=a_{2j}=a_{3j}=0$ for $j=4,5,6$. Then by $A(F)=F$ and direct computation, we have that $a_{17}=a_{27}=a_{37}=0$. Thus, $A$ is of the desired shape. This completes the proof of the lemma.
	\end{proof}

	 Following \cite{OY19}, we recall some definitions about liftability of group actions. Let $F$ be a homogeneous polynomial of degree $d$.  Let $G$ (resp. $\widetilde{G}$) be a finite subgroup of $\PGL(m,\mathbb{C})$ (resp. $\GL(m,\mathbb{C})$). We say  $\widetilde{G}$ is a {\it lifting} of $G$ if $\widetilde{G}$ and $G$ are isomorphic via the natural projection $\pi: \GL(m,\C)\rightarrow \PGL(m,\C).$ We say  $\widetilde{G}$ is an $F$-{\it lifting} of $G$ if $\widetilde{G}$ is a lifting of $G$ and $A(F)=F$, for all $A$ in $\widetilde{G}$. 
	 
	 \begin{lemma}\label{lem:exactseq}
	 Let $F=F(x_1,...,x_m)$ be a smooth form of degree $d$, where $m\ge 4$, $d\ge 3$, $(m,d)\neq (4,4)$. Then ${\rm Aut}(X_F)=\pi (G_F)\subset \PGL(m,\C)$ and there is a short exact sequence of finite groups $$1\rightarrow N \xrightarrow{i} G_F \xrightarrow{\pi| G_F} {\rm Aut}(X_F)\rightarrow 1,$$
	 where $N=\langle \xi_d I_m \rangle\cong C_d$, $i$ is the natural inclusion map, and $\pi | G_F$ is the restriction of $\pi$ to $G_F$.
	 \end{lemma}
	 
	 \begin{proof}
	 By \cite{MM63}, ${\rm Aut}(X_F)=\pi (G_F)$ is a finite group. Clearly $N\subseteq {\rm Ker}(\pi| G_F)$. On the other hand, if $A\in {\rm Ker}(\pi| G_F)$, then $A=\lambda I_m$ for some $\lambda \in \C$, and $F=A(F)=\lambda^d F$. Then $\lambda^d=1$ and $A\in N$. Thus, $N={\rm Ker}(\pi| G_F)$.
	 \end{proof}
	 
	 The following result gives a useful relation among partitionability, $F$-liftability, and structure of automorphism groups of hypersurfaces.
	 
\begin{proposition}\label{thm:relation}
	 Let $F=F(x_1,...,x_m)$ be a smooth form of degree $d$ with $F=H(x_1,...,x_k)+K(x_{k+1},...,x_m)$, where $m\ge 5$, $d\ge 3$, $4\le k\le m-1$, and $(k,d)\neq (4,4)$. Suppose the following statements hold:
	 \begin{enumerate}
  \item  [(1)] ${\rm Aut}(X_H)$ admits an $H$-lifting $\widetilde{{\rm Aut}(X_H)}$;
  \item [(2)] $G_F \subset\{\begin{pmatrix}B&0\\0&C\end{pmatrix}\mid B\in \GL(k,\C), C\in \GL(m-k,\C) \}$.
\end{enumerate}

	 Then ${\rm Aut}(X_F)$ is isomorphic to ${\rm Aut}(X_H)\times G_{K}$ and ${\rm Aut}(X_F)$ has an $F$-lifting $\widetilde{G}$, where 
	 $$\widetilde{G}:=\{\begin{pmatrix}B&0\\0&C\end{pmatrix}\mid B\in \widetilde{{\rm Aut}(X_H)}, C\in G_K\}.$$
	 \end{proposition}
	 
	 \begin{proof}
	By Lemma \ref{lem:exactseq} and the conditions (1) and (2), clearly $\widetilde{G}$ and $\pi (G_F)$ are isomorphic via the natural projection $\pi$. Thus, $\widetilde{G}$ is an $F$-lifting of ${\rm Aut}(X_F)$ and ${\rm Aut}(X_F)\cong \widetilde{G}\cong {\rm Aut}(X_H)\times G_{K}$.
		 \end{proof}
		 
		 Wei--Yu \cite[Theorem 4.11]{WY20} observed that the automorphism group of every smooth cubic threefold is $F$-liftable. Following a similar approach, Gonz\'{a}lez-Aguilera--Liendo--Montero \cite{GLM23} generalized this result to other hypersurfaces.
	
	\begin{thm}[{\cite[Theorem 3.5]{GLM23}}]\label{thm:F-liftable}
		Let $m \geq 3$ and $d \geq 3$ with $(m, d) \neq (3, 3),(4, 4)$. Then the automorphism group of every
		smooth hypersurface $X_F$ of dimension $m-2$ and degree $d$ in $\P^{m-1}$ is $F$-liftable if and only if $d$ and $m$ are relatively prime.
	\end{thm}
	 
	 \begin{example}\label{ex:X2}
	 Let $F:=H+K$, where $H=x_1^3+x_2^3+x_3^3+x_4^3$ and $K=x_5^3+x_6^3+x_7^3+3(\sqrt{3}-1)x_5 x_6 x_7$. Then we have ${\rm Aut}(X_H)\cong C_3^3\rtimes S_4$ and $G_K\cong (C_3^2\rtimes C_3)\rtimes C_4$ (see the proof of Lemma \ref{lem:3vari}). By Theorem \ref{thm:F-liftable} and Proposition \ref{prop:unp+fem}, both (1) and (2) in Proposition \ref{thm:relation} hold. Thus, we have ${\rm Aut}(X_F)\cong (C_3^3\rtimes S_4)\times ((C_3^2\rtimes C_3)\rtimes C_4)$.
	 \end{example}

	\section{Examples and reduction}\label{ss:examples}
	In this section, we give $20$ explicit examples of smooth cubic fivefolds and their automorphisms. Using results on partitionability in the previous section, we determine the automorphism groups of many of them (Theorem \ref{thm:autex}) and completely classify all possible groups acting faithfully on smooth cubic fivefolds defined by polynomials of the form $F=H(x_1,...,x_a)+K(x_{a+1},...,x_{7}), \, 2\le a\le 3$ (Theorem \ref{thm:5+2}).

	\subsection{Examples}\label{mainex}
	
	This subsection lists 20 explicit examples of smooth cubic fivefolds $X_i$ and subgroups $G_{X_i}$ of the automorphism groups ${\rm Aut}(X_i)$ ($i=1,2,...,20$). Since for $i\neq 6, 12, 15,16 ,17,18$, the generators of $G_{X_i}$ are combinations of the following three types of matrices: diagonal matrices, permutations of coordinates, $\frac{1}{\sqrt{3}}\begin{pmatrix}1&1&1\\1&\xi_3&\xi_3^2\\1&\xi_3^2&\xi_3\\ \end{pmatrix}$, we omit them here (the matrix generators of $G_{X_i}\subset \PGL(7,\C)$ for all $i$ except $i=6$ can be found in the ancillary file {\ttfamily Examples4.1.txt} to \cite{YYZ23}).
	In Section \ref{sec:5} we will prove that these 20 groups $G_{X_i}$ classify all $(5,3)$-groups  (Theorem \ref{thm:Main}).
	
	\begin{enumerate}
		\item[(1)] Let $F_{1}=x_1^3+x_2^3+x_3^3+x_4^3+x_5^3+x_6^3+x_7^3$ and $X_1=X_{F_1}$ the Fermat cubic fivefold. Then $\Aut(X_1)=G_{X_1} \cong C_3^6 \rtimes S_7$ and $|G_{X_1}|=2^4 \cdot 3^8 \cdot 5 \cdot 7=3674160$.
		
		\item[(2)] Let $F_{2}=x_1^3+x_2^3+x_3^3+3(\sqrt{3}-1)x_1 x_2 x_3+x_4^3+x_5^3+x_6^3+x_7^3$ and $X_2=X_{F_2}$. Then $G_{X_2} \cong ((C_3^2 \rtimes C_3)\rtimes C_4) \times (C_3^3 \rtimes S_4)$ is a subgroup of $\Aut(X_2)$ and $|G_{X_2}|=2^5\cdot 3^7=69984$. 
		
		\item[(3)] Let $F_{3}=x_1^2x_2+x_2^2x_3+x_3^2x_4+x_4^3+x_5^3+x_6^3+x_7^3$ and $X_3=X_{F_3}$. 
		Then $G_{X_{3}} \cong C_8 \times (C_3^3\rtimes S_3)$ is a subgroup of $\Aut(X_3)$ and $|G_{X_{3}}|=2^4 \cdot 3^4=1296$.
		
		\item[(4)] Let $X_4\subset \mathbb{P}^7$ defined by $x_1^3+x_2^3+x_3^3+x_4^3+x_5^3+x_6^3+x_7^3+x_8^3=x_1+x_2+x_3+x_4+x_5=0$. Then $G_{X_{4}} \cong S_5 \times (C_3^3 \rtimes S_3)$ is a subgroup of $\Aut(X_{4})$ and $|G_{X_{4}}|=2^4\cdot 3^5 \cdot 5=19440$.
		
		\item[(5)] Let $F_5=x_1^2x_2+x_2^2x_3+x_3^2x_4+x_4^2x_5+x_5^3+x_6^3+x_7^3$ and $X_5=X_{F_5}$. 
		Then $G_{X_5} \cong C_{48} \times S_3$ is a subgroup of $\Aut(X_5)$ and $|G_{X_5}|= 2^5 \cdot 3^2 =288$.
		
		\item[(6)] Let $F_{6}=x_1^2x_2+x_2^2x_3+x_3^2x_4+x_4^2x_5+x_5^2x_1+x_6^3+x_7^3$ and $X_6=X_{F_6}$.
		Then $G_{X_6} \cong {\rm PSL}(2,11) \times (C_3^2 \rtimes C_2)$ is a subgroup of $\Aut(X_6)$ and $|G_{X_6}|=2^3 \cdot 3^3 \cdot 5 \cdot 11=11880 $.
		
		\item[(7)] Let $F_{7}=x_1^3+x_2^3+x_3^3+3(\sqrt{3}-1)x_1 x_2 x_3+x_4^3+x_5^3+x_6^3+3(\sqrt{3}-1)x_4 x_5 x_6+x_7^3$ and $X_7=X_{F_7}$. Then $G_{X_{7}} \cong (((C_3^2 \rtimes C_3)\rtimes C_4) \times ((C_3^2 \rtimes C_3)\rtimes C_4))\rtimes C_2 $ is a subgroup of $\Aut(X_{7})$ and $|G_{X_{7}}|=2^5\cdot 3^6=23328$.
		
		\item[(8)] Let $F_{8}=x_1^2x_2+x_2^2x_3+x_3^2x_4+x_4^3+x_5^3+x_6^3+x_7^3+3(\sqrt{3}-1)x_5 x_6 x_7$ and $X_8=X_{F_8}$. Then $G_{X_{8}} \cong ((C_3^2 \rtimes C_3)\rtimes C_4) \times C_8$ is a subgroup of $\Aut(X_{8})$ and $|G_{X_{8}}|= 2^5 \cdot 3^3=864$.
		
		\item[(9)] Let $X_9\subset \mathbb{P}^7$ defined by $x_1^3+x_2^3+x_3^3+x_4^3+x_5^3+x_6^3+x_7^3+x_8^3+3(\sqrt{3}-1)x_6 x_7 x_8=x_1+x_2+x_3+x_4+x_5=0$.  Then $G_{X_{9}} \cong S_5 \times ((C_3^2 \rtimes C_3)\rtimes C_4)$ is a subgroup of $\Aut(X_{9})$ and $|G_{X_{9}}|=2^5\cdot 3^4 \cdot 5=12960 $.
		
		\item[(10)]	Let $F_{10}=x_1^2x_2+x_2^2x_3+x_3^2x_4+x_4^2x_5+x_5^2x_6+x_6^3+x_7^3$ and $X_{10}=X_{F_{10}}$. 
		Then $G_{X_{10}} \cong C_{96}$ is a subgroup of $\Aut(X_{10})$ and $|G_{X_{10}}|=2^5 \cdot 3 = 96$.
		
		\item[(11)]	Let $F_{11}=x_1^2x_2+x_2^2x_3+x_3^2x_4+x_4^2x_5+x_5^2x_6+x_6^2x_1+x_7^3$ and $X_{11}=X_{F_{11}}$. Then $G_{X_{11}} \cong C_{63} \rtimes C_6$ is a subgroup of $\Aut(X_{11})$ and $|G_{X_{11}}|=2 \cdot 3^3 \cdot 7=378 $.
		
		\item[(12)] Let $F_{12}=(x_1^3+x_2^3+x_3^3+x_4^3+x_5^3+x_6^3+x_7^3)+1/5(-3\xi_{24}^7-3\xi_{24}^5+3\xi_{6}-3\xi_{8}+6\xi_{24}-3)\cdot(x_1x_2x_3+x_1x_2x_4+(\xi_{6}-1)x_1x_2x_5+x_1x_2x_6+(\xi_{6}-1)x_1x_3x_4+x_1x_3x_5+x_1x_3x_6+(\xi_{6}-1)x_1x_4x_5-\xi_{6}x_1x_4x_6-\xi_{6}x_1x_5x_6+(\xi_{6}-1)x_2x_3x_4+(\xi_{6}-1)x_2x_3x_5-\xi_{6}x_2x_3x_6+x_2x_4x_5+x_2x_4x_6-\xi_{6}x_2x_5x_6+x_3x_4x_5-\xi_{6}x_3x_4x_6+x_3x_5x_6+x_4x_5x_6)$ and $X_{12}=X_{F_{12}}$.  Then $G_{X_{12}} \cong C_3 . M_{10}$ is a subgroup $\Aut(X_{12})$ and $|G_{X_{12}}|=2^4\cdot 3^3 \cdot 5=2160$ (see \cite{HM19}). Here $M_{10}$ is the Mathieu group of order 720.
		
		\item[(13)] Let $X_{13}\subset \mathbb{P}^7$ defined by $x_1^3+x_2^3+x_3^3+x_4^3+x_5^3+x_6^3+x_7^3+x_8^3=x_1+x_2+x_3+x_4+x_5+x_6+x_7=0$. Then $G_{X_{13}} \cong S_7 \times C_3$ is a subgroup of $\Aut(X_{13})$ and $|G_{X_{13}}|=2^4\cdot 3^3 \cdot 5 \cdot 7=15120$.
		
		\item[(14)]	Let $F_{14}=x_1^2x_2+x_2^2x_5+x_3^2x_4+x_4^2x_5+x_5^2x_6+x_2x_4x_6+x_6^3+x_7^3$ and $X_{14}=X_{F_{14}}$.
		Then $G_{X_{14}} \cong C_3 \times ((C_8 \times C_2)\rtimes C_2)$ is a subgroup of $\Aut(X_{14})$ and $|G_{X_{14}}|=2^5\cdot 3=96 $.
		
		\item[(15)]	Let $F_{15} =x_1^3+8x_2^3+8(-5+4\sqrt{2})x_2x_3^2+2\xi_4(-11+6\sqrt{2})x_3(x_4^2+x_5^2)-4\xi_4x_2((-5+4\sqrt{2})x_4x_5+2(-3+\sqrt{2})x_6x_7)+(1+\xi_4)(-12+11\sqrt{2})(x_5x_6^2-x_4x_7^2)$ and $X_{15}=X_{F_{15}}$. 
		Let $G_{X_{15}}$ be the subgroup of PGL$(7,\C)$ generated by $\Diag(\xi_3,1,-1,\xi_4^3,\xi_4,\xi_8^7,\xi_8)$ and
		$$
		\footnotesize{
			\begin{pmatrix}
				1 & 0 & 0 & 0 & 0 & 0 & 0\\
				0 & -\frac{\sqrt{2}}{4} & 0 & \frac{-3+\sqrt{2}}{8} & \frac{-3+\sqrt{2}}{8}\xi_4 & -\frac{1}{8}+\frac{\xi_4}{4}+\frac{3+\xi_4}{8\sqrt{2}} &-\frac{1}{4}+\frac{\xi_4}{8}-\frac{1+3\xi_4}{8\sqrt{2}}\\
				0 & 0 &\frac{\sqrt{2}}{4} &\frac{3+\sqrt{2}}{8}\xi_4 & \frac{3+\sqrt{2}}{8} & \frac{1}{8}-\frac{\xi_4}{4}+\frac{3+\xi_4}{8\sqrt{2}} & \frac{1}{4}-\frac{\xi_4}{8}-\frac{1+3\xi_4}{8\sqrt{2}}\\
				0 & \frac{1}{2} &\frac{\xi_4}{2} & -\frac{1}{2} &  -\frac{\sqrt{2}}{4}\xi_4 & -\frac{\xi_4}{4}-\frac{\xi_8}{4} & -\frac{\xi_4}{4}+\frac{\xi_8^3}{4}\\
				0 & -\frac{\xi_4}{2} &-\frac{1}{2} &  \frac{\sqrt{2}}{4}\xi_4 & -\frac{1}{2} & -\frac{\xi_4}{4}+\frac{\xi_8}{4} & -\frac{\xi_4}{4}-\frac{\xi_8^3}{4}\\
				0 & \frac{1}{2}+\frac{\xi_8}{2} & -\frac{1}{2}+\frac{\xi_8 }{2}&  \frac{1}{4}+\frac{\xi_8}{4} &\frac{1}{4}-\frac{\xi_8}{4} & \frac{\xi_4}{2} & 0\\
				0 & \frac{\xi_4}{2}+\frac{\xi_8}{2} & -\frac{\xi_4}{2}+\frac{\xi_8}{2} &  -\frac{1}{4}-\frac{\xi_8^3}{4} &-\frac{1}{4}+\frac{\xi_8^3}{4} & 0 & -\frac{\xi_4}{2}
			\end{pmatrix}
		.}
		$$
		Then $G_{X_{15}} \cong C_3 \times ({\rm PSL}(3,2) \rtimes C_2)$ is a subgroup of $\Aut(X_{15})$ and $|G_{X_{15}}|=2^4\cdot 3^2\cdot 7= 1008$. 
		
		\item[(16)]	 Let $F_{16}=x_1^3+x_2^3+x_3^3+\frac{12}{5}x_1x_2x_3+x_1x_4^2+x_2x_5^2+x_3x_6^2+\frac{4\sqrt{15}}{9}x_4x_5x_6+x_7^3$ and $X_{16}=X_{F_{16}}$. Let $G_{X_{16}}$ be the subgroup of PGL$(7,\C)$ generated by 
		$A_{X_{16},1}:=\Diag(1,\xi_3,\xi_3^2,-1,\xi_3,-\xi_3^2,1)$ and $A_{X_{16},2}:=\begin{pmatrix}A_{X_{16},2}'&0\\0&1\end{pmatrix}$, where 
		$$A_{X_{16},2}'=
		\begin{footnotesize}
			\begin{pmatrix}
				 \frac{1}{2} & \frac{1}{2} & \frac{1}{2} & \frac{\sqrt{15}}{18} & \frac{\sqrt{15}}{18} &\frac{\sqrt{15}}{18}\\
				 \frac{1}{2} & \frac{\xi_3}{2} & -\frac{\xi_6}{2} & \frac{\sqrt{15}}{18} & \frac{\sqrt{15}}{18}\xi_3 & -\frac{3\sqrt{5}\xi_4+\sqrt{15}}{36}\\
				\frac{1}{2} & -\frac{\xi_6}{2} & \frac{\xi_3}{2} & \frac{\sqrt{15}}{18} & -\frac{3\sqrt{5}\xi_4+\sqrt{15}}{36} & \frac{\sqrt{15}}{18}\xi_3\\
				 \frac{\sqrt{15}}{10} & \frac{\sqrt{15}}{10}& \frac{\sqrt{15}}{10} & -\frac{1}{2} &-\frac{1}{2} &-\frac{1}{2}\\
				 \frac{\sqrt{15}}{10} & \frac{\sqrt{15}}{10}\xi_3 &-\frac{\sqrt{15}}{10}\xi_6 & -\frac{1}{2} & -\frac{\xi_3}{2}&\frac{\xi_6}{2}\\
				 \frac{\sqrt{15}}{10} & -\frac{\sqrt{15}}{10}\xi_6 &\frac{\sqrt{15}}{10}\xi_3 &-\frac{1}{2} & \frac{\xi_6}{2}&-\frac{\xi_3}{2}\\
			\end{pmatrix}.
		\end{footnotesize}
		$$
		Then $G_{X_{16}} \cong C_3 .A_7$ is a subgroup of $\Aut(X_{16})$ and $|G_{X_{16}}|=2^4\cdot 3^2\cdot 5 \cdot 7= 7560$.

		\item[(17)] Let $F_{17}=x_1^3+x_2^3+x_2x_5^2-\frac{2}{3}x_2 x_5 x_7 + x_2 x_7^2 - \frac{2\xi_6}{3} x_2 x_5 x_6  - 
		\frac{2\xi_6}{3} x_2 x_6 x_7  + (-1 + \xi_6)x_2 x_6^2 +
		x_3^2 x_5 - x_3 x_4 x_5 + x_4^2 x_5 + x_4^2 x_7 + 3\xi_6 x_3 x_4 x_6  + 
		(-1 + \xi_{24} - \xi_8 -  \xi_{24}^5)x_3^2 x_7  + 
		(-1 - 2 \xi_{24} + 2 \xi_8 + 2  \xi_{24}^5) x_3 x_4 x_7  + 
		( \xi_{24} - \xi_6 -  \xi_{24}^7)x_4^2 x_6  + 
		(- \xi_{24} - \xi_6 + \xi_{24}^7)x_3^2 x_6 + x_5^3 - x_6^3 - x_5^2 x_7 - 	x_5 x_7^2 + x_7^3 - \xi_6 x_5^2 x_6+ 2 \xi_6 x_5 x_6 x_7 - 
		\xi_6 x_6 x_7^2 +  (1 - \xi_6) x_5 x_6^2 + (1 - \xi_6)x_6^2 x_7$ and $X_{17}=X_{F_{17}}$. 
		Let $G_{X_{17}}$ be the subgroup of PGL$(7,\C)$ generated by the following three matrices:
		\begin{center}
			\footnotesize{$
				\begin{pmatrix}
					1 & 0 & 0 & 0 & 0 & 0 & 0\\
					0 & 1 & 0 & 0 & 0 & 0 &0\\
					0 & 0 &1-\xi_8-\xi_8^3 &-2& 0 & 0&0\\
					0 & 0 &-1-\xi_8-\xi_8^3 & -1+\xi_8+\xi_8^3&  0 & 0& 0\\
					0 & 0&0 &  0 & 1 & \xi_3^2 & 0\\
					0 & 0 & 0&  0 &0 & -1 & 0\\
					0 & 0 & 0 &  0 &0 & \xi_3^2 & 1 
				\end{pmatrix}$, $
				\begin{pmatrix}
					1 & 0 & 0 & 0 & 0 & 0 & 0\\
					0 & \xi_3 & 0 & 0 & 0 & 0 & 0\\
					0 & 0 & 0 &-\xi_3 & 0 & 0&0\\
					0 & 0 & \xi_3& -\xi_3 &  0 & 0& 0\\
					0 & 0 &0 &  0 & \xi_3 & 1 & 0\\
					0 & 0 & 0&  0 &0 & -\xi_3 & -\xi_{3}^2\\
					0 & 0 & 0 &  0 & 0 & 1 & 0 
				\end{pmatrix}$, $
				\begin{pmatrix}
					1 & 0 & 0 & 0 & 0 & 0 & 0\\
					0 & 1 & 0 & 0 & 0 & 0 & 0\\
					0 & 0 & 1 & -\xi_8-\xi_8^3& 0 & 0&0\\
					0 & 0 & -\xi_8-\xi_8^3 & -1 &  0 & 0& 0\\
					0 & 0& 0 &  0 & 0 & -\xi_3^2&-1\\
					0 & 0 & 0 &  0 &-\xi_3 & 0 & \xi_3\\
					0 & 0 & 0 & 0 & 0 & 0 & -1 
				\end{pmatrix}$.}
		\end{center}
				Then $G_{X_{17}} \cong C_3 \times \GL(2,3)$ is a subgroup of $\Aut(X_{17})$ and $|G_{X_{17}}|= 2^4\cdot 3^2 =144$.

		\item[(18)] Let $F_{18}=x_1^3 + x_2^3 + (\frac{3}{2}\xi_4 - \xi_6 + \frac{1}{2})x_2^2
		x_3 + (-\frac{1}{2}\xi_4 + \frac{1}{2}\xi_6 + \frac{1}{2}\xi_{12} - 1)x_2
		x_3^2 + (-\frac{1}{2}\xi_6 - \frac{1}{2}\xi_{12} + \frac{1}{2})x_3^3 + (\xi_4 - 2\xi_6 - \xi_{12} + 2)x_2^2
		x_4 + (2\xi_{12} - 1)x_2x_3x_4 + (\frac{1}{2}\xi_4 + \frac{1}{2}\xi_6 - \frac{1}{2}\xi_{12})x_3^2
		x_4 + (\xi_4 - 2\xi_6 + 1)x_2x_4^2 + (-\frac{3}{2}\xi_4 + \xi_6 + \xi_{12} - \frac{1}{2})x_3
		x_4^2 + (-\frac{1}{2}\xi_6 + \frac{1}{2}\xi_{12} - \frac{1}{2})x_4^3 + (\xi_4 + \xi_6 - 1)x_2^2
		x_5 + (-\xi_4 - \xi_6 + \xi_{12} - 1)x_2x_3x_5 + (-\frac{1}{2}\xi_4 - \frac{1}{2})x_3^2
		x_5 + (2\xi_{12})x_2x_4x_5 + (\xi_6 - \xi_{12} - 1)x_3x_4
		x_5 + (-\frac{3}{2}\xi_4 + \frac{1}{2}\xi_6 + \frac{3}{2}\xi_{12} - 1)x_4^2x_5 + (-\xi_6 - \xi_{12})x_2
		x_5^2 + (-\frac{1}{2}\xi_4 + \frac{1}{2}\xi_6 + \frac{1}{2}\xi_{12})x_3
		x_5^2 + (\frac{1}{2}\xi_4 + \xi_6 - \xi_{12} - \frac{1}{2})x_4x_5^2 + (-\frac{1}{2}\xi_6 + \frac{1}{2}\xi_{12} + \frac{1}{2})
		x_5^3 + (\xi_{12} - 2)x_2^2x_6 + (-\xi_4 + 2\xi_6 - 1)x_2x_3
		x_6 + (-\frac{1}{2}\xi_6 - \frac{1}{2}\xi_{12} + \frac{1}{2})x_3^2x_6 + (-2\xi_4 + 2\xi_6 + 2\xi_{12} - 2)
		x_2x_4x_6 + (\frac{1}{2}\xi_6 - \frac{1}{2}\xi_{12} + \frac{1}{2})x_4^2x_6 + (-2\xi_4)x_2x_5
		x_6 + (\xi_6 - \xi_{12})x_3x_5x_6 + (\xi_4 - \xi_6 - \xi_{12})x_4x_5
		x_6 + (-\frac{1}{2}\xi_4 + \xi_6 + \xi_{12} - \frac{1}{2})x_5^2x_6 + (\xi_6 - \xi_{12} + 1)x_2
		x_6^2 + (\xi_4 - \frac{3}{2}\xi_6 - \frac{1}{2}\xi_{12} + \frac{1}{2})x_3
		x_6^2 + (\frac{1}{2}\xi_6 - \frac{1}{2}\xi_{12} + \frac{1}{2})x_4
		x_6^2 + (\frac{3}{2}\xi_4 - \frac{1}{2}\xi_6 - \frac{3}{2}\xi_{12} + 1)x_5
		x_6^2 + (-\frac{1}{2}\xi_6 + \frac{1}{2}\xi_{12} - \frac{1}{2})x_6^3 + (\frac{1}{2}\xi_4 - \frac{3}{2}\xi_6 + \frac{1}{2}\xi_{12})
		x_2^2x_7 + (-2\xi_4 + \xi_6 + \xi_{12} - 1)x_2x_3
		x_7 + (\frac{1}{2}\xi_4 - 2\xi_6 - \xi_{12} + \frac{5}{2})x_3^2x_7 + (\xi_6 + \xi_{12} - 2)x_2x_4
		x_7 + (\xi_6 - \xi_{12} - 1)x_3x_4x_7 + (-\frac{1}{2}\xi_4 + \frac{1}{2}\xi_6 + \frac{1}{2}\xi_{12})x_4^2
		x_7 + (-2\xi_4 + 2\xi_{12})x_2x_5x_7 + (-\xi_4 + \xi_6 - \xi_{12})x_3x_5
		x_7 + (-\xi_4 - 1)x_4x_5x_7 + (\frac{3}{2}\xi_6 - \frac{1}{2}\xi_{12} - \frac{1}{2})x_5^2
		x_7 + (2\xi_6 - \xi_{12})x_2x_6x_7 + (\xi_4 - 2\xi_6 + 1)x_3x_6x_7 + (-\xi_{12} + 1)
		x_4x_6x_7 + (2\xi_4 - \xi_6 - 2\xi_{12} + 1)x_5x_6
		x_7 + (-\frac{1}{2}\xi_4 - \xi_6 + \xi_{12} - \frac{1}{2})x_6^2x_7 + (-\frac{1}{2}\xi_4 + \xi_6 - \frac{1}{2})x_2
		x_7^2 + (\frac{1}{2}\xi_4 - \frac{5}{2}\xi_6 + \frac{1}{2}\xi_{12} + 2)x_3
		x_7^2 + (\frac{1}{2}\xi_4 - \xi_{12} + \frac{1}{2})x_4x_7^2 + (-\frac{1}{2}\xi_6 - \frac{1}{2}\xi_{12} + \frac{1}{2})x_5
		x_7^2 + (-\frac{1}{2}\xi_4 - \frac{1}{2}\xi_6 + \frac{1}{2}\xi_{12})x_6
		x_7^2 + (-\frac{1}{2}\xi_6 + \frac{1}{2}\xi_{12} + \frac{1}{2})x_7^3$ and $X_{18}=X_{F_{18}}$. 
			Let $G_{X_{18}}$ be the subgroup of PGL$(7,\C)$ generated by the following three matrices:
			\begin{center}
		\footnotesize{$
			\begin{pmatrix}
				1 & 0 & 0 & 0 & 0 & 0 & 0\\
				0 & 1 & 0 & 0 & \xi_4 & -1&0\\
				0 & 0 &0 &0& 0 & 0&\xi_{12}^7\\
				0 & 0 &0 & 0 &  \xi_{12}^{11} & 0& 0\\
				0 & 0&0 &  0 & 0 & \xi_4^3 & 0\\
				0 & 0 & 0&  \xi_3 &0 & 0 & 0\\
				0 & 0 & \xi_{12}^7 &  0 &0 & 0 & -\xi_3 
			\end{pmatrix}$, $
			\begin{pmatrix}
			1 & 0 & 0 & 0 & 0 & 0 & 0\\
			0 & -\xi_3^2 & \xi_{12}^7 & 0 & 0 & \xi_3^2-\xi_{12}^{11}&-\xi_3\\
			0 & 1-\xi_4 & -\xi_{12}^7 &-\xi_3+\xi_{12}^7 &\xi_4 & 0&0\\
			0 & 0 & \xi_4& 0 &  0 & 0& -1\\
			0 & -\xi_3^2 &0 &  0 & 0 & \xi_3^2 & 0\\
			0 & -\xi_3^2-\xi_{12}^{11} & 0&  0 &0 & \xi_3^2 & -\xi_{12}^7\\
			0 & \xi_4 & \xi_3+\xi_{12}^7 &  -\xi_{12}^7 &-1 & 0 & \xi_{12}^7 
			\end{pmatrix}$, $
			\begin{pmatrix}
			1 & 0 & 0 & 0 & 0 & 0 & 0\\
			0 & -\xi_3^2 & \xi_4 & -\xi_3 & 0 & 0 &\xi_4\\
			0 & -1-\xi_4&-\xi_3 &\xi_3+\xi_{12}^7& 1 & 0&0\\
			0 & -\xi_3 & \xi_4^3 & -1 &  0 & 0& \xi_4^3\\
			0 & 0& -\xi_{12}^{11} &  0 & 0 & 0 & \xi_3^2-\xi_{12}^{11}\\
			0 & 1 & -1+\xi_4&  -\xi_3 &\xi_4 & \xi_3 & -1\\
			0 & 1 & 0 & -\xi_3 & -1 & 0 & -\xi_3 
			\end{pmatrix}$.}
					\end{center}
		Then $G_{X_{18}} \cong (((C_3 \times C_3)\rtimes  C_3)\rtimes Q_8)\rtimes C_3$  is a subgroup of $\Aut(X_{18})$ and $|G_{X_{18}}|= 2^3\cdot 3^4 =648$.
		
		\item[(19)] Let $F_{19}=x_1^2x_2+x_2^2x_3+x_3^2x_4+x_4^2x_5+x_5^2x_6+x_6^2x_7+x_7^3$ and $X_{19}=X_{F_{19}}$. 
		Then $G_{X_{19}} \cong C_{64}$ is a subgroup of $\Aut(X_{19})$ and $|G_{X_{19}}|= 2^6 =64$.

		\item[(20)]	 Let $F_{20}=x_1^2x_2+x_2^2x_3+x_3^2x_4+x_4^2x_5+x_5^2x_6+x_6^2x_7+x_7^2x_1$ and $X_{20}=X_{F_{20}}$ the Klein cubic fivefold. Then $G_{X_{20}} \cong C_{43} \rtimes C_7$ is a subgroup of $\Aut(X_{20})$ and $|G_{X_{20}}|=7 \cdot 43=301$.
		
	\end{enumerate}
	
	\begin{remark}
	For $i\neq 14,15,16,17,18$, the examples $X_i$ and $G_{X_i}$ are easily obtained from known examples of cubic hypersurfaces of dimensions $\leq 4$. For $i= 14,15,16,17,18$, we find  the cubic fivefolds $X_i$ and their symmetries $G_{X_i}$ in the process of our classification of $(5,3)$-groups in Section \ref{sec:5}. Smoothness of these cubics and $G_{X_i}\subseteq {\rm Aut}(X_i)$ can be verified by Mathematica \cite{Wo} and Magma \cite{BCP}.
	\end{remark}
	
	\subsection{Reduction to the examples}
	
	Based on partitionability and the differential method, we have the following
	
	\begin{theorem}\label{thm:autex}
		For $i\in\{1,2,\dots,11, 19,20\}$, $\Aut(X_i)$ is isomorphic to $G_{X_i}$. 
	\end{theorem} 
	
	\begin{proof} 
	
	It is well-known that $\Aut(X_1)= G_{X_1}$.
		For $i\in\{2,\dots,9\}$, by Propositions \ref{prop:unp+fem}, \ref{prop:1+3+3}, \ref{prop:4+3}, we can control the shape of matrices $A$ satisfying $[A]\in {\rm Aut}(X_i)$, and then similar to Example \ref{ex:X2}, we conclude that ${\rm Aut}(X_i)$ is isomorphic to $G_{X_i}$. For $i\in\{10,11, 19, 20\}$, as in proof of \cite[Theorem 3.18]{OY19}, by the differential method (Theorem \ref{thm:dm}), we infer that ${\rm Aut}(X_i)$ are generated by diagonal matrices and permutations of coordinates, which implies that ${\rm Aut}(X_i)=G_{X_i}$.
	\end{proof}
	
	\begin{remark}
 The automorphism groups of most Klein hypersurfaces can be computed using a refinement of the differential method (\cite{GLMV23}). 
	\end{remark}
	
	The following theorem states that if a polynomial defining a cubic hypersurface has a certain type of partition, then the automorphism group of the hypersurface is determined by known results.
	
	\begin{theorem}\label{thm:5+2}
		Let $F=F(x_1,...,x_7)$ be a smooth cubic form. If $F$ has a  $(2,5)$-type  or $(3,4)$-type partition, there exists $i\in\{1,\dots,9\}$ such that $\Aut(X_F)$ is isomorphic to a subgroup of $G_{X_i}$.
	\end{theorem}
	
	\begin{proof}
	
	We sketch the proof since it is similar to Example \ref{ex:X2}. By assumption, $F$ has a maximal $(a_1,...,a_t)$-type partition, where $a_1\le a_2\le...\le a_t$, and we may assume that  $(a_1,...,a_t)$ is one of the following types: $(1,1,5)$, $(3,4)$, $(1,1,1,4)$, $(1,3,3)$, $(1,1,1,1,3)$, $(1,1,1,1,1,1,1)$. If $(a_1,...,a_t)=(1,1,5)$, then we may assume $F=H(x_1,...,x_5)+x_6^3+x_7^3$. By Theorem \ref{thm:F-liftable} and Propositions \ref{thm:relation}, \ref{prop:unp+fem}, we have $${\rm Aut}(X_F)\cong {\rm Aut}(X_H)\times G_{x_6^3+x_7^3}\cong {\rm Aut}(X_H)\times (C_3^2\rtimes C_2).$$ Then by \cite[Theorem 1.1]{WY20} and Theorem \ref{thm:autex}, we infer that ${\rm Aut}(X_F)$ is isomorphic to a subgroup of $G_{X_i}$ for some $i\in\{1,2,...,6\}$. The remaining cases of $(a_1,...,a_t)$ are similar.
	\end{proof}
		Theorem \ref{thm:5+2} has the following direct consequence, which will be frequently used in our classification of $(5,3)$-groups.
	\begin{corollary}\label{cor:partitionbyA}
		Let $F=F(x_1,...,x_7)$ be a smooth cubic form. Suppose that there exists $A\in G_F$ such that $A$ is similar to $\Diag(\xi_3,\xi_3,1,1,1,1,1)$ $($resp. $\Diag(\xi_3,\xi_3,\xi_3,1,1,1,1)$$)$. Then $F$ has a $(2,5)$-type $($resp. $(3,4)$-type$)$ partition. In particular, $\Aut(X_F)$ is isomorphic to a subgroup of $G_{X_i}$ for some  $i\in\{1,\dots,9\}$.
	\end{corollary}

	\section{Automorphism groups of cubic fivefolds}\label{sec:5}
	In this section, we classify all possible subgroups of the automorphism groups of smooth cubic fivefolds (Theorem \ref{thm:Main}). Roughly speaking, we proceed the classification in the following order: abelian subgroups, Sylow subgroups, non-abelian solvable subgroups, non-solvable subgroups. We use Strategy \ref{strategy} to rule out relevant non-abelian groups and it relies on partitionability via abelian subgroups (Theorem \ref{thm:control}) and the notion of special almost $(5,3)$-representations (see Definitions \ref{def:(n,d)-rep}, \ref{def:special}).
	
	Our main theorem is the following:
	\begin{theorem}\label{thm:Main}
		Let $G$ be a finite group. Then the following two conditions are equivalent:
		
		\begin{enumerate}
  \item [(i)] $G$ is isomorphic to a subgroup of one of the $20$ groups $G_{X_i}$ $(i=1,...,20)$ ; and
  \item [(ii)] $G$ acts on a smooth cubic fivefold faithfully.
\end{enumerate}
	\end{theorem}
	
	\begin{remark}\label{rem:Aut=G}
		For $i\in\{1,2,\dots,11, 19,20\}$, in Theorem \ref{thm:autex}, we have  seen that $\Aut(X_i)=G_{X_i}$. Note that $G_{X_i}$ is not isomorphic to a subgroup of $G_{X_j}$ if $1\le i, j\le 20$ and $i\neq j$. Thus, Theorem \ref{thm:Main} implies that $\Aut(X_i)=G_{X_i}$ for $i\in\{12,13,...,18\}$.  The Fermat cubic fivefold has the largest possible order (3674160) for the automorphism group among all smooth cubic fivefolds. A list of all subgroups of the $20$ groups $\Aut(X_i)$ can be found in the ancillary file {\ttfamily 53-groups.txt} to \cite{YYZ23}.

	\end{remark}
	
	\subsection{$(n,d)$-representations and smoothness}\label{ss:smli}

	By Theorem \ref{thm:F-liftable}, the automorphism group of every smooth cubic fivefold $X_F$ admits an $F$-lifting. Thus, a finite group $G$ is a $(5,3)$-group if and only if $G$ admits an injective group homomorphism $\rho$ to $\GL(7,\C)$ such that $\rho(G)$ preserves a smooth cubic form and $\pi \circ \rho$ is injective. Motivated by this, we introduce the following definitions (we identify $\GL(\C^{n+2})$ with ${\rm GL}(n+2,\C)$ by choosing a basis of $\C^{n+2}$).

\begin{definition}\label{def:(n,d)-rep}	
	 Let $A, B\in {\rm GL}(n+2,\C)$.  Let $\rho: G\rightarrow {\rm GL}(n+2,\C)$ be a faithful linear representation of a finite group $G$. Let $\chi$ be the character of $\rho$. 
	 \begin{enumerate}
  \item [(1)] We say $\rho$ is an {\it $(n,d)$-representation} of $G$ if there exists a smooth form $F=F(x_1,...,x_{n+2})$ of degree $d$ such that $\rho(G)$ is an $F$-lifting of a subgroup of ${\rm Aut}(X_F)$. 
  \item  [(2)] We say {\it $\rho$ is an almost $(n,d)$-representation} of $G$ if for every proper abelian subgroup $H<G$, the restriction $\rho| H$ is an $(n,d)$-representation of $H$. 
  \item [(3)] We say $\chi $ is an {\it $(n,d)$-character} (resp. {\it almost $(n,d)$-character}) of $G$ if $\rho$ is an $(n,d)$-representation (resp. almost $(n,d)$-representation).
\end{enumerate}
\end{definition}	
	
	Let $\rho_i: G\rightarrow {\rm GL}(n+2,\C)$ be two faithful linear representations of a finite group $G$ with characters $\chi_i$ $(i=1,2)$. Recall that $\rho_1$ and $\rho_2$ are isomorphic (i.e., there exists $B\in {\rm GL}(n+2,\C)$ such that $B\rho_1(g)B^{-1}=\rho_2(g)$ for all $g\in G$) if and only if $\chi_1=\chi_2$. From this, we have that $\rho_1(G)$ and $\rho_2(G)$ are conjugate in ${\rm GL}(n+2,\C)$ if and only if there exists $f\in {\rm Aut}(G)$ such that $\chi_1=f^*(\chi_2)$. Here $f^*(\chi_2)(g):={\rm Tr}(\rho_2(f(g)))$ for all $g\in G$.	 

		We introduce the following equivalence relation between linear representations of finite groups.
		
\begin{definition}
			Let $\rho_i: G\rightarrow {\rm GL}(n+2,\C)$ be two faithful linear representations of a finite group $G$ with characters $\chi_i$ $(i=1,2)$. Let $d$ be a positive integer.  We say $\rho_1$ and $\rho_2$ are {\it $d$-equivalent} if the groups $\langle \rho_1(G),  \xi_d I_{n+2}\rangle$ and $\langle \rho_2(G),  \xi_d I_{n+2}\rangle$ are conjugate in ${\rm GL}(n+2,\C)$. We say $\chi_1$ and $\chi_2$ are $d$-equivalent if $\rho_1$ and $\rho_2$ are $d$-equivalent.
\end{definition}
		The motivation of the definition of $d$-equivalence is the following 
		
\begin{lemma}\label{lem:d-equiv}
			Let $\rho_i: G\rightarrow {\rm GL}(n+2,\C)$ be $d$-equivalent faithful linear representations of a finite group $G$ $(i=1,2)$. If $\rho_1$ is an $(n,d)$-representation, so is $\rho_2$. 
\end{lemma}
		
\begin{proof}
			By assumption, there exists a smooth form $F$ of degree $d$ such that $A_1(F)=F$ for all $A_1\in \langle \rho_1(G),  \xi_d I_{n+2}\rangle$. Since $B\langle \rho_1(G),  \xi_d I_{n+2}\rangle B^{-1}=\langle \rho_2(G),  \xi_d I_{n+2}\rangle$ for some $B\in {\rm GL}(n+2,\C)$, we have $A_2(H)=H$ for all $A_2\in \langle \rho_2(G),  \xi_d I_{n+2}\rangle$, where $H:=B^{-1}(F)$. Because
				$|\langle \rho_2(G), \xi_d I_{n+2}\rangle| = |\langle \rho_1(G), \xi_d I_{n+2}\rangle| = d|G|$,
				one has $\rho_2(G) \cap\langle \xi_d I_{n+2}\rangle= \{I_{n+2}\}$ and so $\pi \rho_2$ is injective. Thus, $\rho_2(G)$ is an $H$-lifting of $\pi\rho_2(G)\subset {\rm Aut}(X_{H})$, which implies  $\rho_2$ is an $(n,d)$-representation.
\end{proof}
	
	Using smoothness in an effective way is important for classification of automorphism groups of smooth hypersurfaces. We mainly use two ``combinatorial" non-smoothness tests Lemmas \ref{lem:nsmcubic} and \ref{lem:cri} in our classification of $(5,3)$-groups. 
		
	\begin{lemma}[{\cite[Lemma 3.2 and Proposition 3.3]{OY19}}]\label{lem:nsmcubic}
		Let $F=F(x_1,\dots,x_7)$ be a homogeneous polynomial of degree 3. Then $F$ is not smooth if one of the following four conditions is true:
		
		\begin{enumerate}
  \item [(i)] There exists $ 1\leq i\leq 7$, such that for all $ 1\leq j\leq 7,  x_i^2x_j\notin F$;
  \item [(ii)] There exist three distinct variables $x_p,x_q,x_r$, such that  $F\in ( x_p,x_q,x_r) $;
  \item [(iii)] There exist four distinct variables $x_p,x_q,x_r,x_s$, such that $F\in ( x_p,x_q) +( x_r,x_s) ^2$; 
  \item [(iv)] There exist five distinct variables $x_p,x_q,x_r,x_s,x_t$, such that $F\in ( x_p) +(x_q, x_r,x_s,x_t) ^2$.
\end{enumerate}
		\end{lemma}

	It turns out that this test is often convenient and sufficient to rule out relevant candidates of cubics and groups for our purposes.
	
	\begin{example}\label{ex:C2}
	Let $\rho: C_2\rightarrow\GL(7,\C)$ be a faithful linear representation of $C_2$. Then we may assume $\rho(C_2)=\langle A_{a,b}\rangle$, where $A_{a,b}:=\begin{pmatrix}-I_a&0\\0&I_b\end{pmatrix}$, $(a,b)\in \{(1,6)$, $(2,5)$, $(3,4)$, $(4,3)$, $(5,2)$, $(6,1)\}$. Suppose $F$ is a smooth cubic form satisfying $A(F)=F$. If $a\ge 4$, then $F\in ( x_5,x_6,x_7)$, which contradicts Lemma \ref{lem:nsmcubic}. On the other hand, $A_{1,6}$, $A_{2,5}$, $A_{3,4}$ preserve the smooth cubic form $x_1^2x_4+x_2^2x_5+x_3^2x_6+x_4^3+x_5^3+x_6^3+x_7^3$. Thus, up to isomorphisms, $C_2$ has exactly three $(5,3)$-representations corresponding to $A_{1,6}, A_{2,5}, A_{3,4}$ respectively.
		\end{example}

	Lemma \ref{lem:nsmcubic} fails in some cases and we use the following stronger test in our computer program.
	\begin{lemma}[{\cite[Lemma 1.6]{GLM23}}]\label{lem:cri}
		Let $F=F(x_1,\dots,x_7)$ be a homogeneous polynomial of degree 3. If there exist three mutually disjoint collections of variables $V_1$, $V_2$, $V_3$ such that $\bigcup_{i} V_i=\{x_1,\dots,x_7\}$, $|V_1|>|V_2|$ and for every monomial $\mathfrak{m}\in F$, $\mathfrak{m}$ can be expressed in one of the following forms:
		
		\begin{enumerate}
  \item [(i)] $\mathfrak{m}=x_px_qx_r$, $x_p,x_q\in V_1$ and $x_r\in V_2$;
  \item [(ii)] $\mathfrak{m}=x_px_qx_r$, $x_p\in V_1$ and $x_q,x_r\in V_2\cup V_3$; or
  \item [(iii)] $\mathfrak{m}=x_px_qx_r$, $x_p,x_q,x_r\in V_2\cup V_3$;
\end{enumerate}
		then $F$ is not smooth.
	\end{lemma}
		
	\subsection{Abelian $(5,3)$-groups}\label{ss:Ab}
	We classify abelian $(5,3)$-groups in this subsection. 
	As a direct consequence of \cite[Theorem 1.3]{GL13}, we have the following
	\begin{proposition}\label{pp:porder}
		Let $g$ be an element of primary order in a $(5,3)$-group. Then ${\rm ord}(g)=2^a,3^b,5,7,11$ or $43$, where $a,b>0$.
	\end{proposition}

	The Sylow $p$-subgroups of $(5,3)$-groups with $p>3$ are known.
	\begin{proposition}[{cf. \cite[Example 4.3]{GLM23}}]\label{prop:sylowp}
		Let $G$ be a nontrivial $p$-group with $p\in\{5,7,11,43\}$. If $G$ is a $(5,3)$-group,  then $G\cong C_p$.
	\end{proposition}

	Let $\mathcal{G}_a$ be the set of subgroups of the following $29$ abelian groups:
	$C_{43}$, $C_{9}\times C_5$, $C_5 \times C_4 \times C_3$, $C_9 \times C_7$, $C_{64}$, $C_{11}\times C_3 \times C_2$, $C_{9} \times C_8$, $C_{9}^2$, $C_5 \times C_3^2\times C_2$, $C_{32}\times C_3$, $C_{16}\times C_{3} \times C_2$, $C_8 \times C_4 \times C_3$, $C_{11}\times C_3^2$, $C_{9} \times C_4 \times C_3$, $C_{9} \times C_3 \times C_2^2$, $C_{5}\times C_3^3$, $C_{16}\times C_{3}^2$, $C_{8} \times C_3^2 \times C_2$, $C_4^2 \times C_3^2$, $C_{4} \times C_3^2 \times C_2^2$, $C_{9} \times C_3^2 \times C_2$, $C_{8} \times C_3^3$, $C_{4} \times C_3^3 \times C_2$, $C_3^3 \times C_2^3$, $C_{9} \times C_3^3$, $C_{4} \times C_3^4$, $C_3^4 \times C_2^2$, $C_{3}^5 \times C_2$, $C_3^6$. 
	
	By computing abelian subgroups of $G_{X_i}$, we have the following
	
	\begin{lemma}\label{lem:ab53}
	An abelian group $G$ is in $\mathcal{G}_a$ if and only if $G$ is isomorphic to an abelian subgroup of $G_{X_i}$ for some $i\in \{1,2,...,20\}$. In particular, all groups in $\mathcal{G}_a$ are $(5,3)$-groups.
	\end{lemma}
	
 We say $A\in \GL(m, \C)$ is {\it semi-permutation} if $A$ is a diagonal matrix up to permutation of columns, or equivalently, $A$ has exactly $m$ nonzero entries. 
  
Similar to Example \ref{ex:C2}, using Lemmas \ref{lem:nsmcubic} and \ref{lem:cri}, we can compute $(5,3)$-representations, up to $3$-equivalence, of groups in $\mathcal{G}_a$ with the help of computer algebra. Based on this, we obtain the following result.	
	
	\begin{theorem} \label{thm:control}
		If a $(5,3)$-group $G$ contains a subgroup isomorphic to one of the following groups:
		$C_{11}$, $C_9 \times C_3$, $C_9 \times C_4$, $C_9 \times C_2^2$,  $C_{43}$, $C_9 \times C_5$, $C_5 \times C_3^2$, $C_5 \times C_4 \times C_3$, $C_9 \times C_7$, $C_{64}$, $C_3^2 \times C_2^3$, $C_4 \times C_3^2 \times C_2$, $C_3^4$, $C_{32}\times C_3$, $C_{16} \times C_3 \times C_2$, $C_8 \times C_4 \times C_3$, $C_4 \times C_3^3$, $C_3^3 \times C_2^2$, $C_{16}\times C_3^2$, 
		then $G$ is isomorphic to a subgroup of $G_{X_i}$ for some $i\in\{1,2,\dots,11,19,20\}$.

	\end{theorem}
	
	\begin{proof}
             First we explain the strategies of the proof. Let $X_F$ be a smooth cubic fivefold defined by $F$. Suppose $G< {\rm Aut}(X_F)$ contains $G_0$, where $G_0$ is one of the $19$ abelian groups in the list. We classify $(5,3)$-representations $\rho$ of $G_0$ up to $3$-equivalence. We may assume that the matrices in $\rho(G_0)$ are diagonal since $G_0$ is abelian. By computing all cubic monomials preserved by $\rho(G_0)$, we infer that either $F$ has an $(a_1,a_2)$-type partition, where $(a_1,a_2)=(2,5),(3,4)$, or $G_{F}$ consists of semi-permutation matrices by the differential method. Then by Theorems \ref{thm:autex} and \ref{thm:5+2}, we conclude that ${\rm Aut}(X_{F})$ is isomorphic to a subgroup of ${\rm Aut}(X_i)=G_{X_i}$ for some $i\in\{1,2,\dots,11,19,20\}$. We give details for two typical cases $G_0=C_{11}, C_9\times C_5$ and the other cases are similar (representatives of all $3$-equivalence classes of $(5,3)$-representations of the $19$ abelian groups can be found in the ancillary file {\ttfamily Theorem5.12.txt} to \cite{YYZ23}).
             
             Case $G_0=C_{11}$: Up to $3$-equivalence, $G_0$ has exactly one $(5,3)$-representation $\rho$ given by $\rho (G_0)=\langle A \rangle $, where $ A=\Diag(\xi_{11}^9,\xi_{11}^5,\xi_{11}^4,\xi_{11}^3,\xi_{11},1,1) $. Then the set of cubic monomials preserved by $A$ consists of $x_7^3$, $x_6 x_7^2$, $x_6^2 x_7$, $x_6^3$, $x_3^2 x_4$, $x_2 x_4^2$, $x_2^2 x_5$, $x_1 x_5^2$, $x_1^2 x_3$. Thus, by $A(F)=F$, we have that $F$ can be partitioned into $H(x_1,x_2,x_3,x_4,x_5)+K(x_6,x_7)$. By Theorem \ref{thm:5+2}, ${\rm Aut}(X_{F})$ is isomorphic to a subgroup of $G_{X_i}$ for some  $i\in\{1,\dots,6\}$.

		Case $G_0=C_9\times C_5$: Up to $3$-equivalence, $G_0$ has exactly one $(5,3)$-representation $\rho$ given by $\rho(G_0)=\langle A, B\rangle$, where $A=\Diag(\xi_{9},\xi_{9}^{7},\xi_9^4,1,1,1,1)$ and $B=\Diag(1,1,1,\xi_5,\xi_5^3,\xi_5^4,\xi_5^2)$. Then the cubic monomials preserved by $A$ and $B$ are $x_1^2x_2$, $x_2^2x_3$, $x_3^2x_1$, $x_4^2x_5$, $x_5^2x_6$, $x_6^2x_7$, $x_7^2x_4$. Since $F$ is preserved by $\rho(G_0)$, $F$ can be partitioned into $F=H(x_1,x_2,x_3)+K(x_4,x_5,x_6,x_7)$, and by Theorem \ref{thm:5+2}, there exists $i\in\{1,\dots,9\}$ such that $\Aut(X_F)$ is isomorphic to a subgroup in $G_{X_i}$.
	\end{proof}
	
	Now we are ready to classify abelian $(5,3)$-groups.
	\begin{thm}\label{thm:abelian}
		Let $G$ be an abelian group. Then $G$ is a $(5,3)$-group if and only if $G\in \mathcal{G}_a$.
	\end{thm}
	
	\begin{proof}
	It suffices to show that if $G$ is an abelian group such that (1) all of its proper subgroups are in $\mathcal{G}_a$ and (2) $G\notin \mathcal{G}_a$, then $G$ is not a $(5,3)$-group. Then by Lemma \ref{lem:ab53} and Theorem \ref{thm:control}, we are reduced to rule out the following $13$ groups: $C_{27}$, $C_2^4$, $C_8\times C_2^2$, $C_4^2\times C_2$, $C_{32}\times C_2$, $C_{16}\times C_4$, $C_8^2$, $C_7\times C_2$, $C_5\times C_2^2$, $C_5\times C_7$, $C_8\times C_5$, $C_4\times C_2^2\times C_3$, $C_7\times C_3^2$. Then the idea is to prove that these groups have no $(5,3)$-representation by considering $(5,3)$-representations of their proper subgroups. We give details for $C_7\times C_2$ and the other cases are similar.
	
	Suppose $\rho$ is a $(5,3)$-representation of $G=C_7\times C_2$ with $\rho(G)$ preserving a smooth cubic form $F$. Note that $C_7$ has only one $(5,3)$-representation up to $3$-equivalence given by $A_7:=\Diag(\xi_{7}^6,\xi_{7}^5,\xi_{7}^4,\xi_{7}^3,\xi_{7}^2,\xi_7,1) $ (representatives of all $3$-equivalence classes of $(5,3)$-representations of cyclic groups of primary orders can be found in the ancillary file {\ttfamily Theorem5.13.txt} to \cite{YYZ23}). By considering the restriction $\rho | C_7$, we may assume that $\rho(G)=\langle A_7, A_2\rangle$, where $A_2=$ $\Diag(\pm 1,$ $\pm 1,...,\pm 1)$. The cubic monomials preserving by $A_7$ are $x_7^3$, $x_4x_5^2$, $x_4^2x_6$, $x_2x_6^2$, $x_2^2x_3$, $x_1x_3^2$, $x_1^2x_5$, $x_3x_5x_6$, $x_3x_4x_7$, $x_2x_5x_7$, $x_1x_6x_7$,  $x_1x_2x_4$. By Lemma \ref{lem:nsmcubic}, the first $7$ monomials must be in $F$. From this, by $A_2(F)=F$, we deduce that $A_2=\Diag(1, 1,1,1,1,1,1)$, which is a contradiction since $\rho(G)\cong C_7\times C_2$. \end{proof}

	\subsection{Solvable $(5,3)$-groups}\label{sec:solva}
	In this subsection, a classification of solvable $(5,3)$-groups is presented. The classification begins with identifying all $2$-groups and $3$-groups followed by examining other solvable $(5,3)$-groups of order less than or equal to $2^5\cdot 3^4\cdot 5\cdot 7$. 
	
	\begin{definition}\label{def:special}
	Let $\rho: G\rightarrow {\rm GL}(7,\C)$ be a faithful linear representation of a finite group $G$ with character $\chi$. We say $\rho$ (resp. $\chi$) is a {\it special} representation (resp. {\it special} character) if there exists no $A\in \rho(G)$ such that $A$ is similar to either $\Diag(\xi_3^{a+1},\xi_3^{a+1},\xi_3^a,\xi_3^a,\xi_3^a,\xi_3^a,\xi_3^a)$ or $\Diag(\xi_3^{a+1},\xi_3^{a+1},\xi_3^{a+1},\xi_3^a,\xi_3^a,\xi_3^a,\xi_3^a)$ for some $a\in \{0,1,2\}$.
	\end{definition}
	
	 As a direct consequence of Corollary \ref{cor:partitionbyA}, the following lemma plays an important role in our strategy of ruling out groups.
	
	\begin{lemma}\label{lem:rulingout}
		Let $G$ be a finite group. Then $G$ is not a $(5,3)$-group if the following two conditions are satisfied:
		
		\begin{enumerate}
  \item [(1)] There exists no $1\le i\le 20$ such that $G$ is isomorphic to a subgroup of $G_{X_i}$;
  \item [(2)] $G$ admits no special $(5,3)$-representation.
 \end{enumerate}
	\end{lemma}

We use the following strategy to classify non-abelian $(5,3)$-groups.

\begin{strategy}\label{strategy}
Let $m$ be a positive integer. Suppose that all $(5,3)$-groups of orders $m^\prime<m$ satisfying $m^\prime \mid m$ have been found. We classify non-abelian $(5,3)$-groups of order $m$ as follows.

	\textbf{Step 1}: We compute the (finite) set $\mathcal{B}_m$ of  non-abelian groups $G$ of order $m$ satisfying the following conditions: 
	
	\begin{enumerate}
	\item [(1)] All proper subgroups of  $G$ are $(5,3)$-groups;
	
	\item [(2)] $G_{X_i}$ has no subgroup isomorphic to $G$ for all $1\le i\le 20$;
	
	\item [(3)] $G$ contains none of the $19$ abelian groups in Theorem \ref{thm:control}.
	\end{enumerate}
	If $\mathcal{B}_m=\emptyset$, then we are done. Otherwise, we do case-by-case check for groups in $\mathcal{B}_m$. For each $G\in \mathcal{B}_m$, go to \textbf{Step 2}.	
	
	\textbf{Step 2}:  Compute the (finite) set $\mathcal{R}_G$ of the $3$-equivalence classes of the special almost $(5,3)$-characters of $G$. If $\mathcal{R}_G=\emptyset$, then $G$ is ruled out by Lemmas \ref{lem:d-equiv} and \ref{lem:rulingout}. Otherwise, go to \textbf{Step 3}.
	
	\textbf{Step 3}: For each $\chi\in \mathcal{R}_G$, (i) we compute a representation $\rho$ affording $\chi$; (ii) we compute the cubic forms $F$ invariant by all matrices in $\rho (G)$; (iii) we prove that such forms $F$ are not smooth. 
	\end{strategy}	
	
	Here we explain why Strategy \ref{strategy} works out. Let $G$ be a non-abelian group of order $m$ such that $G_{X_i}$ has no subgroup isomorphic to $G$ for all $1\le i\le 20$. It suffices to show that $G$ is not a $(5,3)$-group if one of the following three statements is true: (a) $G\notin \mathcal{B}_m$; (b) $\mathcal{R}_G=\emptyset$; (c) each $\chi\in \mathcal{R}_G$ is not a $(5,3)$-character. Note that by assumption, (2) in \textbf{Step 1} holds for $G$. Thus, if $G\notin \mathcal{B}_m$, then $G$ does not satisfy either (1) or (3) in \textbf{Step 1}, which implies that $G$ is not a $(5,3)$-group by Theorem \ref{thm:control}. From now on, we may assume that $G\in \mathcal{B}_m$. Suppose $G$ is a $(5,3)$-group. Then by Lemma \ref{lem:rulingout}, $G$ has a special $(5,3)$-representation $\rho$. This representation gives rise to a special $(5,3)$-character $\chi$, which is of course a special almost $(5,3)$-character in $\mathcal{R}_{G}$, which implies that $\mathcal{R}_G\neq \emptyset$ and $\chi\in \mathcal{R}_{G}$ is a $(5,3)$-character (i.e., both (b) and (c) fail).
	
	\begin{remark}
		For orders $|G|=m$ coprime to $3$,  all representations of $G$ are special. However, for cases $3\mid m$, $G$ might have many non-special almost $(5,3)$-representations. Thus, focusing on special almost $(5,3)$-representations in \textbf{Step 2} and \textbf{Step 3} reduces the amount of calculations in our classification considerably.
	\end{remark}

\begin{remark}\label{rem:subtest}
	Following \cite{OY19}, we call the condition (1) in \textbf{Step 1} the sub-test. Like in \cite{OY19} and \cite{WY20}, for relevant group orders $m$, the sub-test often rules out most of the finite groups which are not $(5,3)$-groups.
\end{remark}

\begin{remark}
	The idea of ruling out candidate $(n,d)$-groups via restricting their characters to (abelian) subgroups is used in previous studies, e.g. \cite[Lemma 6.11]{OY19} for $(n,d)=(3,5)$ and \cite[Theorem 6.1]{WY20} for $(n,d)=(3,3)$.  In order to handle a larger list of relevant candidates of hypersurfaces and groups in the cases $(n,d)=(5,3), (4,3)$, we use this idea in a more systematic way (e.g., introducing new notions including (special) almost $(n,d)$-character and $d$-equivalence). For computation in the steps of Strategy \ref{strategy}, we use a mixture of GAP \cite{GAP}, Mathematica \cite{Wo}, and Sage \cite{Sage}. 
	The codes needed in Strategy \ref{strategy} are contained in the following ancillary files to \cite{YYZ23}: {\ttfamily GAPsubtest-53groups.txt}, {\ttfamily LHsnewLHsnewred.txt}, {\ttfamily GAPcodesnewforMath.txt}, {\ttfamily CubicFivefolds.m}.
	\end{remark}                

As illustrations, we apply Strategy \ref{strategy} to determine Sylow subgroups of $(5,3)$-groups.
 \begin{thm}\label{thm:sylow23}
		Let $G$ be a $p$-group with $p=2,3$. If $G$ is a $(5,3)$-group, then $G$ is isomorphic to a subgroup of $G_{X_i}$ for some $1\le i\le 20$.
	\end{thm}

	\begin{proof}
		We give proof for 2-groups, and the classification for 3-groups is similar. Let $m=|G|=2^a$. 
		
		Cases $a=1,2,3$: Each group of order $2^a$ is a subgroup of one of the 20 groups $G_{X_i}$. In particular, $\mathcal{B}_{2}=\mathcal{B}_{4}=\mathcal{B}_{8}=\emptyset$ and we are done.
		
		Case $a=4$: After running \textbf{Step 1}, we have $\mathcal{B}_{16}=\{C_2\times Q_8\}.$  By \textbf{Step 2}, we find $\mathcal{R}_{C_2\times Q_8}=\emptyset$, which implies that $C_2\times Q_8$ is not a $(5,3)$-group.
		
		Cases $a=5,6$: Similar to case $a=4$.
		
		Cases $a>6$: Note that $C_{64}$ is the only $(5,3)$-group of order $64$. Since any finite group of order $2^a$ contains a subgroup of order $64$. Thus, by Theorem \ref{thm:control}, there is no $(5,3)$-group of order $2^a$. This completes the classification for $2$-groups.
	\end{proof}
	
	Next we give an upper bound for the orders of candidate $(5,3)$-groups which we are reduced to consider from now on.
	\begin{proposition}\label{prop:boundorder}
		Let $G$ be a $(5,3)$-group. If $|G|$ does not divide $90720=2^5\cdot 3^4\cdot 5\cdot 7$, then $G$ is isomorphic to a subgroup of $G_{X_i}$ for some $1\le i\le 20$. In particular, if $F=F(x_1,x_2,...,x_7)$ is a smooth cubic form having neither $(2,5)$ nor $(3,4)$-type partition, then $|\Aut(X_F)|\le 90720$.
	\end{proposition}
	
	\begin{proof}
		By Propositions \ref{pp:porder}, \ref{prop:sylowp} and Theorem \ref{thm:sylow23}, we have 
		$$|G|=2^{a_2}\cdot 3^{a_3}\cdot 5^{a_5}\cdot 7^{a_7}\cdot 11^{a_{11}}\cdot 43^{a_{43}},$$ 
		where $a_2\le 6$, $a_3\le 8$, $a_5\le 1$, $a_7\le 1$, $a_{11}\le 1$, $a_{43}\le 1$. 
		By Theorem \ref{thm:sylow23}, any $(5,3)$-group of order $243=3^5$ contains either $C_9\times C_3$ or $C_3^4$. Since $|G|$ does not divide $2^5\cdot 3^4\cdot 5\cdot 7$, it follows that $G$ contains one of the following groups: $C_{64}$, $C_9\times C_3$, $C_3^4$, $C_{11}$, $C_{43}$. Then by Theorem \ref{thm:control}, $G$ is isomorphic to a subgroup of $G_{X_i}$ for some $1\le i\le 20$.
	\end{proof}
	
	Now we are ready to classify all solvable $(5,3)$-groups.
	
	\begin{theorem}\label{thm:solvable}
		Let $G$ be a solvable $(5,3)$-group. Then there exists $i\in\{1,\dots,20\}$ such that $G$ is isomorphic to a subgroup of $G_{X_i}$. 
	\end{theorem}
	
	\begin{proof}
		By Proposition \ref{prop:boundorder}, we are reduced to consider cases $|G|=m$ dividing $90720$. Then as in the proof of Theorem \ref{thm:sylow23}, we use Strategy \ref{strategy} to proceed the classification. It turns out that all relevant candidate groups which are not $(5,3)$-groups can be ruled out by Strategy \ref{strategy} (the details are included in the ancillary file {\ttfamily Theorem5.22.txt} to \cite{YYZ23}). Since the arguments are completely similar as before, we only give an example for which \textbf{Step 3} is involved. Consider the case $|G|=m=96$ (the outputs of our computer-aided calculations for this case are contained in the ancillary file {\ttfamily Example-m96-53-groups.txt} to \cite{YYZ23}). By \textbf{Step 1}, $\mathcal{B}_{96}$ consists of $4$ groups. Applying \textbf{Step 2}, it turns out $C_3\rtimes C_{32}\in \mathcal{B}_{96}$ is the only one for which $\mathcal{R}_G$ is not empty. In fact, $\mathcal{R}_{C_3\rtimes C_{32}}$ has exactly one element, say $\chi$. The image $\rho(C_3\rtimes C_{32})$ of the special almost $(5,3)$-representation $\rho$ affording $\chi$ is generated by $A_1=\begin{pmatrix}S  & 0 \\0& T\end{pmatrix}$ and	$A_2=\Diag(1,1,1,1,1,\xi_3,\xi_3^2)$, where $S=\Diag(1,\xi_4^3,\xi_8^5,\xi_{16}^3,\xi_{32}^{13})$, $T=\begin{pmatrix}0&1\\1&0\end{pmatrix}$. Thus, $\{x_1^3$, $x_1x_6x_7$, $x_4x_5^2$, $x_6^3+x_7^3$, $x_3x_4^2$, $x_2x_3^2\}$ is a basis of the homogenous polynomials of degree $3$ preserved by $A_1$ and $A_2$. Then if $F$ is a cubic form invariant by $\rho(C_3\rtimes C_{32})$, we have $x_2^2x_j\notin F$ for all $1\le j\le 7$, which implies $F$ is not smooth. Therefore, $C_3\rtimes C_{32}$ is not a $(5,3)$-group.
	\end{proof}

	\subsection{Proof of Theorem \ref{thm:Main}}\label{thmproofmain} In this subsection, we prove our main theorem (Theorem \ref{thm:Main}). By Theorem \ref{thm:solvable} and Proposition \ref{prop:boundorder}, it suffices to prove that if $G$ is a non-solvable $(5,3)$-groups with $m:=|G|$ dividing $2^5\cdot 3^4 \cdot 5 \cdot 7$, then $G$ is isomorphic to a subgroup of $G_{X_i}$ for some $i\in \{1,2,...,20\}$. 
			
		Cases $m\leq 2000$: We use Strategy \ref{strategy} to handle these cases.
		 In fact, it turns out that by \textbf{Step 1}, $\mathcal{B}_m=\emptyset$ (resp. $\mathcal{B}_m=\{ C_2 \times A_6\}$) for $m\neq 720$ (resp. $m=720$). By \textbf{Step 2}, we find $\mathcal{R}_{C_2 \times A_6}=\emptyset$, which implies $C_2 \times A_6$ is not a $(5,3)$-group.
		 
		Cases $m>2000$: It is well-known that a non-abelian finite simple group of order dividing $2^5\cdot 3^4 \cdot 5 \cdot 7$ is one of the following: $A_5$, $\PSL(3,2)$, $A_6$, $\PSL(2,8)$, $A_7$, ${\rm PSU}(3,3)$. Then $m$ must be divided by the order of one of the $6$ simple groups. Thus, it suffices to consider non-solvable groups of the following orders:	$2016$, $2160$,  $2520$, $3024$, $3240$, $3360$, $3780$, $4320$, $4536$, $5040$, $6048$, $6480$, $7560$, $9072$, $10080$, $11340$, $12960$, $15120$, $18144$, $22680$, $30240$, $45360$, $90720$. As above, we also apply Strategy \ref{strategy} to treat these orders. For cases $m=2160$, $2520$, $3024$,  $3240$, $3360$, $3780$, $4536$, $5040$, $7560$, $15120$, we completely use computer to do computation in the steps of Strategy \ref{strategy} like in the cases $m\leq 2000$ (the details can be found in the ancillary file {\ttfamily Theorem5.1.txt} to \cite{YYZ23}). For the remaining orders, we use a more theoretical approach based on our classification of $(5,3)$-groups of smaller orders. Next we give the details for a typical case $m=2016$.
						
		Suppose $G$ is a non-solvable $(5,3)$-group of order $2016$. Let $N$ be a maximal proper normal subgroup of $G$. Then we consider the following short exact sequence:
		$$1\longrightarrow N \longrightarrow G \longrightarrow M \longrightarrow 1,$$
		where $M\cong G/N$ is a simple group. Then $M$ is one of the following groups: $C_2$, $C_3$, $C_7$, $\PSL(3,2)$, $\PSL(2,8)$. 
						
			(1) If $M\cong C_2$, then $|N|=1008$ and $N\cong C_3\times ($PSL$(3,2)\rtimes C_2)$ by our classification of $(5,3)$-groups of order $1008$. Note that $N$ contains a unique subgroup, say $H$, of order 336. Moreover, $H\cong  {\rm PSL}(3,2)\rtimes C_2$. Thus for any $g\in G$, $gHg^{-1}\subset gNg^{-1}=N$, which implies that $gHg^{-1}=H$ is a normal subgroup of $G$. We get another short exact sequence:
			$$1\longrightarrow H \longrightarrow G \longrightarrow M' \longrightarrow 1,$$
			where $|M'|=6$. So there exists a subgroup $H'< M^\prime$ with $|H^\prime|=2$. Then $G$ has a subgroup of order $336\cdot 2=672$, which is impossible since there is no $(5,3)$-group of order $672$.
			
			(2) If $M\cong C_3$, then $|N|=672$, which is impossible.
			
			(3) If $M\cong C_7$, then both $N$ and $M$ are solvable, which contradicts non-solvability of $G$.
			
			(4) If $M\cong \PSL(3,2)$, then $|N|=12$. Since $|M|=2^3\cdot 3\cdot 7$, there exists a subgroup $H<G$ such that $|H|=84$, which is impossible by previous classification.
			
			(5) If $M\cong \PSL(2,8)$, then $|N|=4$. Since $|M|=2^3\cdot 3^2\cdot 7$, there exists a subgroup $H<G$ such that $|H|=28$, similarly, it is impossible.
			
			Therefore, we conclude that there is no non-solvable $(5,3)$-group of order $2016$. The remaining cases for $m$ can be handled similarly. This completes the proof of Theorem \ref{thm:Main}.

	\section{Automorphism groups of cubic fourfolds}\label{ss:fourfolds}
	In this section, we classify all groups faithfully acting on smooth cubic fourfolds based on the classification of $(5,3)$-groups. It turns out that there are $15$ maximal groups among them (Theorem \ref{thm:fourfold}). As a by-product, we find explicit defining polynomials of two cubic fourfolds with maximal symplectic automorphism groups isomorphic to $A_7$ and $M_{10}$ respectively (Theorems \ref{thm:A7} and \ref{thm:M10}).

	\subsection{Examples}\label{ex:4folds}
	The $15$ smooth cubic fivefolds $X_j$ ($j=1,2,...,7, 10,11,...,15, 17,18$) in Subsection \ref{mainex} are defined by smooth cubic forms admitting partitions of $(6,1)$-type and we have seen the explicit description of their automorphism groups ${\rm Aut}(X_j)=G_{X_j}$ (see Remark \ref{rem:Aut=G}). From this and Lemma \ref{lem:exactseq}, we immediately obtain the following $15$ examples of smooth cubic fourfolds $X_{i}^{\prime}=X_{F_{i}^{\prime}}$ ($i=1,2,...,15$ respectively) via the relation $\widehat{F_{i}^{\prime}}=F_j$ (up to obvious permutations of variables) and the explicit description of their automorphism groups ${\rm Aut}(X_{F_{i}^{\prime}})$ (in fact, if $A\in G_{F_{i}^{\prime}}$, then $\begin{pmatrix}A&0\\0&1\end{pmatrix}\in G_{\widehat{F_{i}^{\prime}}}$; the matrix generators of ${\rm Aut}(X_{F_{i}^{\prime}})$ can be found in the ancillary file {\ttfamily Examples6.1.txt} to \cite{YYZ23}).

	\begin{enumerate} 
		\item[(1)] Let $F_{1}^{\prime}=x_1^3+x_2^3+x_3^3+x_4^3+x_5^3+x_6^3$ and $X_1'=X_{F_{1}^{\prime}}$ the Fermat cubic fourfold.
		Then $\Aut(X_1') \cong C_3^5 \rtimes S_6$ of order $2^4 \cdot 3^7 \cdot 5=174960$.
		
		\item[(2)] Let $F_{2}^{\prime}=x_1^3+x_2^3+x_3^3+3(\sqrt{3}-1)x_1 x_2 x_3+x_4^3+x_5^3+x_6^3$ and $X_2'=X_{F_{2}^{\prime}}$.
		Then $\Aut(X_2') \cong ((C_3 \times (C_3^3\rtimes C_3))\rtimes C_3) \rtimes (C_4 \times C_2) $ of order $2^3\cdot 3^6=5832$. 
		
		\item[(3)] Let $F_{3}^{\prime}=x_1^2x_2+x_2^2x_3+x_3^2x_4+x_4^3+x_5^3+x_6^3$ and $X_3'=X_{F_{3}^{\prime}}$. 
		Then $\Aut(X'_{3}) \cong C_8 \times (C_3^2 \rtimes C_2)$ of order $2^4 \cdot 3^2=144$.
		
		\item[(4)] Let $X_4'\subset \P^6$ defined by $x_1^3+x_2^3+x_3^3+x_4^3+x_5^3+x_6^3+x_7^3=x_1+x_2+x_3+x_4+x_5=0$.
		Then $\Aut(X'_{4}) \cong S_5 \times (C_3^2 \rtimes C_2)$ of order $2^4\cdot 3^3 \cdot 5=2160$. 
		
		\item[(5)] Let $F_{5}^{\prime}=x_1^2x_2+x_2^2x_3+x_3^2x_4+x_4^2x_5+x_5^3+x_6^3$ and $X_5'=X_{F_{5}^{\prime}}$.  
		Then $\Aut(X'_{5}) \cong C_{48}$ of order $3 \cdot 2^4 =48$.
		
		\item[(6)]	Let $F_{6}^{\prime}=x_1^2x_2+x_2^2x_3+x_3^2x_4+x_4^2x_5+x_5^2x_1+x_6^3$ and $X_6'=X_{F_{6}^{\prime}}$. Then $\Aut({X'_6}) \cong {\rm PSL}(2,11) \times C_3$ of order $2^2 \cdot 3^2 \cdot 5 \cdot 11=1980$. 
		
		\item[(7)] Let $F_{7}^{\prime}=x_1^3+x_2^3+x_3^3+3(\sqrt{3}-1)x_1 x_2 x_3+x_4^3+x_5^3+x_6^3+3(\sqrt{3}-1)x_4 x_5 x_6$ and $X_7'=X_{F_{7}^{\prime}}$. 
		Then $\Aut({X'_{7}}) \cong ((C_3 \times(C_3^2\rtimes C_3))\rtimes C_3)\rtimes (C_4^2 \rtimes C_2)$ of order $2^5\cdot 3^5=7776$.
		
		\item[(8)] Let $F_{8}^{\prime}=x_1^2x_2+x_2^2x_3+x_3^2x_4+x_4^2x_5+x_5^2x_6+x_6^3$ and $X_8'=X_{F_{8}^{\prime}}$.
		Then $\Aut({X'_{8}}) \cong C_{32}$ of order $2^5= 32$.
		
		\item[(9)] Let $F_{9}^{\prime}=x_1^2x_2+x_2^2x_3+x_3^2x_4+x_4^2x_5+x_5^2x_6+x_6^2x_1$ and $X_9'=X_{F_{9}^{\prime}}$.
		Then $\Aut({X'_{9}}) \cong C_{21} \rtimes C_6$ of order $2 \cdot 3^2 \cdot 7=126$.
		
		\item[(10)] Let $F_{10}^{\prime}=(x_1^3+x_2^3+x_3^3+x_4^3+x_5^3+x_6^3)+\frac{1}{5}(-3\xi_{24}^7-3\xi_{24}^5+3\xi_{6}-3\xi_{8}+6\xi_{24}-3)\cdot(x_1x_2x_3+x_1x_2x_4+(\xi_{6}-1)x_1x_2x_5+x_1x_2x_6+(\xi_{6}-1)x_1x_3x_4+x_1x_3x_5+x_1x_3x_6+(\xi_{6}-1)x_1x_4x_5-\xi_{6}x_1x_4x_6-\xi_{6}x_1x_5x_6+(\xi_{6}-1)x_2x_3x_4+(\xi_{6}-1)x_2x_3x_5-\xi_{6}x_2x_3x_6+x_2x_4x_5+x_2x_4x_6-\xi_{6}x_2x_5x_6+x_3x_4x_5-\xi_{6}x_3x_4x_6+x_3x_5x_6+x_4x_5x_6)$ and $X_{10}'=X_{F_{10}^{\prime}}$.
		Then $\Aut({X'_{10}}) \cong M_{10}$ of order $2^4\cdot 3^2 \cdot 5=720$.
		
		\item[(11)] Let $X'_{11}\subset \P^{6}$ defined by $x_1^3+x_2^3+x_3^3+x_4^3+x_5^3+x_6^3+x_7^3=x_1+x_2+x_3+x_4+x_5+x_6+x_7=0$.
		Then $\Aut({X'_{11}}) \cong S_7$ of order $2^4\cdot 3^2 \cdot 5\cdot7=5040 $. 
		
		\item[(12)] Let $F_{12}^{\prime}=x_1^2x_2+x_2^2x_5+x_3^2x_4+x_4^2x_5+x_5^2x_6+x_2x_4x_6+x_6^3$ and $X_{12}'=X_{F_{12}^{\prime}}$.
		Then $\Aut({X'_{12}}) \cong(C_8 \times C_2)\rtimes C_2$ of order $2^5=32$.
		
		\item[(13)] Let $F_{13}^{\prime}=8x_1^3+8(-5+4\sqrt{2})x_1x_2^2+2\xi_4(-11+6\sqrt{2})x_2(x_3^2+x_4^2)-4\xi_4x_1((-5+4\sqrt{2})x_3x_4+2(-3+\sqrt{2})x_5x_6)+(1+\xi_4)(-12+11\sqrt{2})(x_4x_5^2-x_3x_6^2)$ and $X_{13}'=X_{F_{13}^{\prime}}$. Then $\Aut({X'_{13}}) \cong {\rm PSL}(3,2) \rtimes C_2$ of order $2^4\cdot 3\cdot 7= 336$. 
		\item[(14)] Let $F'_{14}=x_1^3+x_1x_4^2-\frac{2}{3}x_1 x_4 x_6 + x_1 x_6^2 - \frac{2\xi_6}{3} x_1 x_4 x_5  - 
		\frac{2\xi_6}{3} x_1 x_5 x_6  + (-1 + \xi_6)x_1 x_5^2 +
		x_2^2 x_4 - x_2 x_3 x_4 + x_3^2 x_4 + x_3^2 x_6 + 3\xi_6 x_2 x_3 x_5  + 
		(-1 + \xi_{24} - \xi_8 -  \xi_{24}^5)x_2^2 x_6  + 
		(-1 - 2 \xi_{24} + 2 \xi_8 + 2  \xi_{24}^5) x_2 x_3 x_6  + 
		( \xi_{24} - \xi_6 -  \xi_{24}^7)x_3^2 x_5  + 
		(- \xi_{24} - \xi_6 + \xi_{24}^7)x_2^2 x_5 + x_4^3 - x_5^3 - x_4^2 x_6 - 
		x_4 x_6^2 + x_6^3 - \xi_6 x_4^2 x_5+ 2 \xi_6 x_4 x_5 x_6 - 
		\xi_6 x_5 x_6^2 +  (1 - \xi_6) x_4 x_5^2 + (1 - \xi_6)x_5^2 x_6$ and $X_{14}'=X_{F_{14}^{\prime}}$. 
		Then $\Aut({X'_{14}}) \cong \GL(2,3)$ of order $2^4\cdot 3 =48$.
		
		\item[(15)] Let $F'_{15}=  x_1^3 + (\frac{3}{2}\xi_4 - \xi_6 + \frac{1}{2})x_1^2
		x_2 + (-\frac{1}{2}\xi_4 + \frac{1}{2}\xi_6 + \frac{1}{2}\xi_{12} - 1)x_1
		x_2^2 + (-\frac{1}{2}\xi_6 - \frac{1}{2}\xi_{12} + \frac{1}{2})x_2^3 + (\xi_4 - 2\xi_6 - \xi_{12} + 2)x_1^2
		x_3 + (2\xi_{12} - 1)x_1x_2x_3 + (\frac{1}{2}\xi_4 + \frac{1}{2}\xi_6 - \frac{1}{2}\xi_{12})x_2^2
		x_3 + (\xi_4 - 2\xi_6 + 1)x_1x_3^2 + (-\frac{3}{2}\xi_4 + \xi_6 + \xi_{12} - \frac{1}{2})x_2
		x_3^2 + (-\frac{1}{2}\xi_6 + \frac{1}{2}\xi_{12} - \frac{1}{2})x_3^3 + (\xi_4 + \xi_6 - 1)x_1^2
		x_4 + (-\xi_4 - \xi_6 + \xi_{12} - 1)x_1x_2x_4 + (-\frac{1}{2}\xi_4 - \frac{1}{2})x_2^2
		x_4 + (2\xi_{12})x_1x_3x_4 + (\xi_6 - \xi_{12} - 1)x_2x_3
		x_4 + (-\frac{3}{2}\xi_4 + \frac{1}{2}\xi_6 + \frac{3}{2}\xi_{12} - 1)x_3^2x_4 + (-\xi_6 - \xi_{12})x_1
		x_4^2 + (-\frac{1}{2}\xi_4 + \frac{1}{2}\xi_6 + \frac{1}{2}\xi_{12})x_2
		x_4^2 + (\frac{1}{2}\xi_4 + \xi_6 - \xi_{12} - \frac{1}{2})x_3x_4^2 + (-\frac{1}{2}\xi_6 + \frac{1}{2}\xi_{12} + \frac{1}{2})
		x_4^3 + (\xi_{12} - 2)x_1^2x_5 + (-\xi_4 + 2\xi_6 - 1)x_1x_2
		x_5 + (-\frac{1}{2}\xi_6 - \frac{1}{2}\xi_{12} + \frac{1}{2})x_2^2x_5 + (-2\xi_4 + 2\xi_6 + 2\xi_{12} - 2)
		x_1x_3x_5 + (\frac{1}{2}\xi_6 - \frac{1}{2}\xi_{12} + \frac{1}{2})x_3^2x_5 + (-2\xi_4)x_1x_4
		x_5 + (\xi_6 - \xi_{12})x_2x_4x_5 + (\xi_4 - \xi_6 - \xi_{12})x_3x_4
		x_5 + (-\frac{1}{2}\xi_4 + \xi_6 + \xi_{12} - \frac{1}{2})x_4^2x_5 + (\xi_6 - \xi_{12} + 1)x_1
		x_5^2 + (\xi_4 - \frac{3}{2}\xi_6 - \frac{1}{2}\xi_{12} + \frac{1}{2})x_2
		x_5^2 + (\frac{1}{2}\xi_6 - \frac{1}{2}\xi_{12} + \frac{1}{2})x_3
		x_5^2 + (\frac{3}{2}\xi_4 - \frac{1}{2}\xi_6 - \frac{3}{2}\xi_{12} + 1)x_4
		x_5^2 + (-\frac{1}{2}\xi_6 + \frac{1}{2}\xi_{12} - \frac{1}{2})x_5^3 + (\frac{1}{2}\xi_4 - \frac{3}{2}\xi_6 + \frac{1}{2}\xi_{12})
		x_1^2x_6 + (-2\xi_4 + \xi_6 + \xi_{12} - 1)x_1x_2
		x_6 + (\frac{1}{2}\xi_4 - 2\xi_6 - \xi_{12} + \frac{5}{2})x_2^2x_6 + (\xi_6 + \xi_{12} - 2)x_1x_3
		x_6 + (\xi_6 - \xi_{12} - 1)x_2x_3x_6 + (-\frac{1}{2}\xi_4 + \frac{1}{2}\xi_6 + \frac{1}{2}\xi_{12})x_3^2
		x_6 + (-2\xi_4 + 2\xi_{12})x_1x_4x_6 + (-\xi_4 + \xi_6 - \xi_{12})x_2x_4
		x_6 + (-\xi_4 - 1)x_3x_4x_6 + (\frac{3}{2}\xi_6 - \frac{1}{2}\xi_{12} - \frac{1}{2})x_4^2
		x_6 + (2\xi_6 - \xi_{12})x_1x_5x_6 + (\xi_4 - 2\xi_6 + 1)x_2x_5x_6 + (-\xi_{12} + 1)
		x_3x_5x_6 + (2\xi_4 - \xi_6 - 2\xi_{12} + 1)x_4x_5
		x_6 + (-\frac{1}{2}\xi_4 - \xi_6 + \xi_{12} - \frac{1}{2})x_5^2x_6 + (-\frac{1}{2}\xi_4 + \xi_6 - \frac{1}{2})x_1
		x_6^2 + (\frac{1}{2}\xi_4 - \frac{5}{2}\xi_6 + \frac{1}{2}\xi_{12} + 2)x_2
		x_6^2 + (\frac{1}{2}\xi_4 - \xi_{12} + \frac{1}{2})x_3x_6^2 + (-\frac{1}{2}\xi_6 - \frac{1}{2}\xi_{12} + \frac{1}{2})x_4
		x_6^2 + (-\frac{1}{2}\xi_4 - \frac{1}{2}\xi_6 + \frac{1}{2}\xi_{12})x_5
		x_6^2 + (-\frac{1}{2}\xi_6 + \frac{1}{2}\xi_{12} + \frac{1}{2})x_6^3$ and $X_{15}'=X_{F_{15}^{\prime}}$. 
		Then $\Aut({X'_{15}}) \cong ((C_3 \times C_3)\rtimes Q_8)\rtimes C_3$  of order $2^3\cdot 3^3 =216$.
	\end{enumerate}

	The defining equations of the examples $X_i'$ ($i=1,2,...,11$) are known (see e.g., \cite{HM19}, \cite{LZ22}, \cite{Zh22}). It seems that the cubic fourfolds $X_i'$ ($i=12,13,14,15$) are new.

	\subsection{$C_3$-covering groups and classification}
	
	\begin{theorem}\label{thm:fourfold}
		Let $G$ be a finite group. Then the following two conditions are equivalent:
		
		\begin{enumerate}
  \item [(i)] $G$ is isomorphic to a subgroup of one of the $15$ groups ${\rm Aut}(X_i')$ $($$i=1,2,...,15$$)$; and
  \item [(ii)] $G$ acts on a smooth cubic fourfold faithfully.
\end{enumerate}
	\end{theorem}

		 A list of all subgroups of the $15$ groups $\Aut(X'_i)$ is contained in the ancillary file {\ttfamily 43-groups.txt} to \cite{YYZ23}. We will prove Theorem \ref{thm:fourfold} based on close relations between $(4,3)$-groups and $(5,3)$-groups. First we make some reduction based on partitionability as in our classification of $(5,3)$-groups. The proof of the following result is similar to that of Theorem \ref{thm:5+2} and we omit the details.
	
	\begin{theorem}\label{thm:4parti}
		Let $X$ be a smooth cubic fourfold defined by the homogeneous polynomial $F$. If $F$ is partitionable, then there exists $i\in\{1,\dots 7\}$ such that $\Aut (X)$ is isomorphic to a subgroup of $\Aut(X'_i)$.
	\end{theorem}

To study relations between $(4,3)$-groups and $(5,3)$-groups in a more general setting, we introduce the following definition.
	
	\begin{definition}\label{def:Cdcovering}
	Let $G$ and $\widehat{G}$ be two finite groups. Let $d$ be a positive integer. We say $\widehat{G}$ is a {\it $C_d$-covering group} of $G$ if  the centre $Z(\widehat{G})$ of $\widehat{G}$ contains a subgroup $N$ such that $N\cong C_d$ and $\widehat{G}/N\cong G$.
	\end{definition}

	Now we have the following
	
	\begin{lemma}\label{lem:widehatF}
	Let $F=F(x_1,...,x_m)$ be a smooth form of degree $d$ with $m\ge 3$, $d\ge 3$. We define $\widehat{F}:=F+x_{m+1}^d$. Then
	
	\begin{enumerate}
  \item [(1)] $\widehat{F}$ is a smooth form of degree $d$.
  \item [(2)] If $G_{\widehat{F}}$ consists of semi-permutation matrices, then so does $G_F$.
  \item [(3)] If $d=3$, then $F$ is partitionable if and only if $\widehat{F}$ has an $(a_1,a_2)$-type partition with $a_1\ge 2$, $a_2\ge 2$. 
\end{enumerate}
	\end{lemma}
	
	\begin{proof}
	(1) and (2) follow from definitions. To prove (3), from now on, we assume that $d=3$. If $F$ is partitionable, then we may assume that $F$ has a partition of type $(a,b)$ with $a\ge 1$, $b\ge2$, which implies that $\widehat{F}$ has a partition of type $(a+1,b)$ with $a+1\ge 2$, $b\ge 2$. On the other hand, if $\widehat{F}$ has an $(a_1,a_2)$-type partition with $a_1\ge 2$, $a_2\ge 2$, then there exists $A\in \GL(m+1,\C)$ such that $$A(\widehat{F})=H(x_1,...,x_{a_1})+K(x_{a_1+1},...,x_{a_1+a_2}).$$
	For positive integers $n_1,n_2$, we set $B_{n_1,n_2}:=\begin{pmatrix} I_{n_1}&0\\0&\xi_3 I_{n_2}\end{pmatrix}$. Then $B_{a_1,a_2}\in G_{A(\widehat{F})}$. Note that $G_{A(\widehat{F})}=A^{-1} G_{\widehat{F}} A$. Then by Lemma \ref{lem:linear} and Proposition \ref{prop:unp+fem}, we have $A B_{a_1,a_2} A^{-1}\in G_{\widehat{F}}$ and $A B_{a_1,a_2} A^{-1}=\begin{pmatrix} B &0\\0&\lambda \end{pmatrix}$, where $B\in G_F$ and $\lambda\in \C$. Since $B_{a_1,a_2}$ and $A B_{a_1,a_2} A^{-1}$ have the same eigenvalues (counting multiplicities), we have that $B$ is similar to either $B_{a_1-1,a_2}$ or $B_{a_1,a_2-1}$. From this, we conclude that $F$ is partitionable (see Corollary \ref{cor:partitionbyA}).
	\end{proof}

	 Every $(n,d)$-group has at least one $C_d$-covering $(n+1,d)$-group, which gives strong constraints on $(n,d)$-groups. 
	
	\begin{lemma}\label{lem:(4,3)and(5,3)}
	Let $G$ be a subgroup of ${\rm Aut}(X_F)$, where $F=F(x_1,...,x_{n+2})$ is a smooth form of degree $d$, where $n\ge 2$, $d\ge 3$, $(n,d)\neq (2,4)$.  Let $\widehat{F}$ denote the degree $d$ smooth form $F+x_{n+3}^d$. Then the following statements hold:
	\begin{enumerate}
  \item [(1)] There exists a subgroup $G'\subset \GL(n+3,\C)$ such that: (i) $G'$ is a $C_d$-covering $(n+1,d)$-group of $G$, (ii) $G'$ is an $\widehat{F}$-lifting of $\pi(G')\subseteq {\rm Aut}(X_{\widehat{F}})$, and (iii) $G'$ contains the matrix $\begin{pmatrix}\xi_d I_{n+2}&0\\0&1\end{pmatrix}$;

  \item [(2)] If $G$ admits an $F$-lifting, then $G\times C_d$ is an $(n+1,d)$-group. In particular, if $d$ is a prime number not dividing $|G|$, then $G\times C_d$ is an $(n+1,d)$-group.
\end{enumerate}
	\end{lemma}
	
	\begin{proof}
	 Recall that the center $Z(G_F)$ of $G_F=\{A\in \GL(n+2,\C)\mid A(F)=F\}$ contains $N:=\langle\xi_d I_{n+2} \rangle\cong C_d$. By Lemma \ref{lem:exactseq}, we have the short exact sequence of groups
	$$1\rightarrow N \xrightarrow{i} G_F \xrightarrow{\pi|G_F} {\rm Aut}(X_F)\rightarrow 1,$$
	where $i$ is the natural inclusion map. Let $\widehat{G}< G_F$ be the pre-image of the subgroup $G\subseteq {\rm Aut}(X_F)$ under the map $\pi|G_F$. Then $N\subseteq Z(\widehat{G})$ and $\widehat{G}/N\cong G$, which means that $\widehat{G}$ is a $C_d$-covering group. Let $G^\prime:=\{ \begin{pmatrix}A&0\\0&1\end{pmatrix} \mid A\in \widehat{G}\}$. Clearly $\widehat{G}$ is isomorphic to $G^\prime\cong \pi (G')\subseteq \Aut(X_{\widehat{F}})$. From this, we conclude the statement (1). 
	
	If $G$ admits an $F$-lifting, say $\widetilde{G}$, then 
	$$\widehat{G}=\widetilde{G}\times N\cong G\times C_d,$$ which implies the first sentence in (2). Then by \cite[Theorem 4.8]{OY19}, the second sentence in (2) holds. This completes the proof of the lemma.	\end{proof}
	
	\begin{example}\label{ex:43by53}
	By Theorem \ref{thm:Main}, $C_{64}$ is the only $(5,3)$-group of order $64$ and $C_{64}\times C_3$ is not a $(5,3)$-group. Then by Lemma \ref{lem:(4,3)and(5,3)},  there is no $(4,3)$-group of order $64$.
	\end{example}
	
	From now on, we focus on $(4,3)$-groups and their $C_3$-covering $(5,3)$-groups.
	
	\begin{theorem}\label{thm:ruleout43}
	Let $G$ be a finite group. Suppose that for every $C_3$-covering $(5,3)$-group $\widehat{G}$ of $G$, one of the following statements holds: 
	
	\begin{enumerate}
  \item [(i)] $\widehat{G}$ contains one of the $19$ abelian groups in Theorem \ref{thm:control};
  \item [(ii)] $\widehat{G}$ has no special $(5,3)$-representation $\rho$ with $\rho(\widehat{G})$ containing $\Diag(\xi_3,...,\xi_3,1)$.
\end{enumerate}
	
	If $G$ is a $(4,3)$-group, then ${\rm Aut}(X_i')$ contains a subgroup isomorphic to $G$ for some $i\in\{1,2,...,9\}$.
	\end{theorem}
	
	\begin{proof}
	 Suppose $G<\Aut(X_F)$, where $F=F(x_1,...,x_6)$ is a smooth cubic form.  Let $G'$ be as in Lemma \ref{lem:(4,3)and(5,3)} (1). In particular, $G'$ is a $C_3$-covering $(5,3)$-group of $G$. If (ii) holds for $G'$, then there exists $A\in G_{\widehat{F}}$ such that $A$ is similar to either $\Diag(\xi_3,\xi_3,1,1,1,1,1)$ or $\Diag(\xi_3,\xi_3,\xi_3,1,1,1,1)$, which implies the theorem by Corollary \ref{cor:partitionbyA}, Lemma \ref{lem:widehatF}, Theorem \ref{thm:4parti}. If (i) holds for $G'$, then like in the proof of Theorem \ref{thm:control}, we conclude that ${\rm Aut}(X_i')$ contains a subgroup isomorphic to $G$ for some $i\in\{1,2,...,9\}$.  This completes the proof of the theorem.
	\end{proof}

	 By adapting Strategy \ref{strategy}, we use the following strategy to classify $(4,3)$-groups.

\begin{strategy}\label{strategy43}
Let $m$ be a positive integer. Suppose that all $(4,3)$-groups of orders $m^\prime<m$ satisfying $m^\prime \mid m$ have been found. We classify $(4,3)$-groups of order $m$ as follows.

	\textbf{Step 1}: We compute the (finite) set $\mathcal{B}_m'$ of groups $G$ of order $m$ satisfying the following conditions: 
	\begin{enumerate}
  \item [(1)] All proper subgroups of  $G$ are $(4,3)$-groups;
  \item [(2)] $\Aut({X_i'})$ has no subgroup isomorphic to $G$ for all $1\le i\le 15$.
\end{enumerate}	
	If $\mathcal{B}_m'=\emptyset$, then we are done. Otherwise, we do case-by-case check for groups in $\mathcal{B}_m'$. For each $G\in \mathcal{B}_m'$, go to \textbf{Step 2}.	
	
	\textbf{Step 2}: Compute the (finite) set $\mathcal{C}_G$ defined as follows: if $3\mid m $, $$\mathcal{C}_G:=\{\widehat{G}\mid \widehat{G} \text{ is a }C_3\text{-covering } (5,3)\text{-group of }G\};$$
	if $3\nmid m $,  $\mathcal{C}_G:=\{\widehat{G}\mid \widehat{G}\cong C_3\times G\text{ and } \widehat{G}\text{ is a }(5,3)\text{-group}\}$. If $\mathcal{C}_G=\emptyset$, then $G$ is ruled out (Lemma \ref{lem:(4,3)and(5,3)}). Otherwise, go to \textbf{Step 3}.
	
	\textbf{Step 3}: For each $\widehat{G}\in \mathcal{C}_G$, we prove that either (i) $\widehat{G}$ contains one of the $19$ groups in Theorem \ref{thm:control} by computing abelian subgroups of $\widehat{G}$ or (ii) $\widehat{G}$ has no special $(5,3)$-representation $\rho$ with $\rho(\widehat{G})$ containing $\Diag(\xi_3,...,\xi_3,1)$
by computing $\mathcal{R}_{\widehat{G}}$ (and applying \textbf{Step 3} of Strategy \ref{strategy} to each $\chi\in \mathcal{R}_{\widehat{G}}$ if $\mathcal{R}_{\widehat{G}}\neq\emptyset$). Then by Theorem \ref{thm:ruleout43}, $G$ is ruled out.
	\end{strategy}	
	
	Next we prove Theorem \ref{thm:fourfold} based on Strategy \ref{strategy43}.
	
	\begin{proof}[Proof of Theorem \ref{thm:fourfold}]
		By Theorem \ref{thm:Main} and Lemma \ref{lem:(4,3)and(5,3)} (see Example \ref{ex:43by53}), the order of a $(4,3)$-group is of the following form 
		$$2^{a_2}\cdot 3^{a_3}\cdot 5^{a_5}\cdot 7^{a_7}\cdot 11^{a_{11}},$$ 
		where $a_2\le 5$, $a_3\le 7$, $a_5\le 1$, $a_7\le 1$, $a_{11}\le 1$.  Then by Theorem \ref{thm:ruleout43}, we are reduced to classify $(4,3)$-groups of orders dividing $2^{5}\cdot 3^{3}\cdot 5\cdot 7$ (see the proof of Proposition \ref{prop:boundorder}).

		  Similar to the proof of Theorem \ref{thm:Main} based on Strategy \ref{strategy}, we use Strategy \ref{strategy43} to rule out groups inductively in the sense of increasing orders $m$ of relevant groups. We only give the details for a typical case $m=16$ (other cases can be found in the ancillary file {\ttfamily Theorem6.1.txt} to \cite{YYZ23}; the outputs of our computer-aided calculations for the case $m=16$ are contained in the ancillary file {\ttfamily Example-m16-43-groups.txt}). 
		  
		  By \textbf{Step 1}, $\mathcal{B}_{16}'$ consists of $5$ groups: $C_2\times Q_8$, $C_2^4$, $(C_4\times C_2)\rtimes C_2$, $C_4\rtimes C_4$, $C_4\times C_2^2$. In \textbf{Step 2}, for the first $2$ (resp. the last $3$) groups $G$ in the list, we have $\mathcal{C}_G=\emptyset$ (resp. $\{C_3\times G\}$). Thus, $C_2\times Q_8$ and $C_2^4$ are not $(4,3)$-groups. By \textbf{Step 3}, $\mathcal{R}_{\widehat{G}}=\emptyset$ for $\widehat{G}=C_3\times ((C_4\times C_2)\rtimes C_2)$, $C_3\times (C_4\rtimes C_4)$, which implies that $(C_4\times C_2)\rtimes C_2$ and $C_4\rtimes C_4$ are not $(4,3)$-groups. Applying \textbf{Step 3} to $\widehat{G}=C_3\times C_4\times C_2^2$, we have $\mathcal{R}_{\widehat{G}}$ contains only one element $\chi$ and the image $\rho(\widehat{G})$ of the special almost $(5,3)$-representation $\rho$ affording $\chi$ is generated by $\Diag(-1,\xi_4,1,1,1,1,1)$ and $\Diag(1,1,-1,1,1,1,1)$,	  $\Diag(1,1,1,-1,1,1,1)$ and $\Diag(1,1,\xi_3,\xi_3^2,\xi_3^2,\xi_3,1)$. From this, we conclude that $\widehat{G}$ satisfies (ii) in Theorem \ref{thm:ruleout43}, which implies that $G=C_4\times C_2^2$ is not a $(4,3)$-group. Thus, we complete the proof for the case $m=16$. Similarly, we can handle all other cases $m\mid 2^{5}\cdot 3^{3}\cdot 5\cdot 7$. This completes the proof of the theorem.
	\end{proof}
	
	\begin{remark}
	Automorphisms of cubic fourfolds naturally induce automorphisms of their Fano varieties of lines which are hyperk\"{a}hler manifolds of $K3^{[2]}$-type. In particular, all $(4,3)$-groups can act faithfully on hyperk\"{a}hler manifolds of $K3^{[2]}$-type. It would be interesting to apply our classification of $(4,3)$-groups to study (fixed point loci of) finite groups of automorphisms of hyperk\"{a}hler manifolds of $K3^{[2]}$-type.
	\end{remark}
	
	\subsection{Symplectic automorphism groups} The {\it symplectic automorphism group} $\Aut^s(X_F)$ of a smooth cubic fourfold $X_F$ consists of the symplectic automorphisms $f$ of $X_F$ (i.e., the induced action on $H^{3,1}(X_F)\cong \C$ is trivial). If the defining equation $F$ and matrix generators of $\Aut(X_F)$ are explicitly given, one can directly compute $\Aut^s(X_F)$ via the following result.
	
	\begin{lemma}[{\cite[Lemma 3.2]{Fu16}}]\label{lem:symcon}
		Let $X$ be a smooth cubic fourfold defined by $F(x_1,\dots,x_6)$. Let $f=[A]$ be an element in $\Aut(X)$, $A\in \GL(6,\mathbb{C})$, with ${\rm ord}(f)={\rm ord}(A)$ and $A(F)=\lambda F$ with $\lambda\in \C$, then  $f$ is symplectic if and only if $\det(A)=\lambda^2$.
	\end{lemma}
	
	\begin{example}
	The automorphism group ${\rm Aut}(X_5')\cong C_{48}$ of $X_5'$ is generated by the matrix $A_{X_5'}:={\rm diag}(\xi_{16},\xi_8^7, \xi_4 ,-1,1,\xi_3)$. Since $A_{X_5'}(F_5')=F_5'$ and ${\rm det}(A_{X_5'})$ is a $48$-th primitive root of unity, by Lemma \ref{lem:symcon}, we have that ${\rm Aut}^s(X_5')$ is trivial. More generally, if $\Aut(X_F)$ admits an $F$-lifting $\widetilde{\Aut(X_F)}\subset \GL(6,\C)$, then $$\Aut^s(X_F)\cong (\widetilde{\Aut(X_F)}\cap {\rm SL}(6,\C)).$$ 
	\end{example}
	
	The symplectic automorphism groups $\Aut^s(X_i')$ $(i=1,2,...,15)$ can be computed similarly and the result is summarized in the Table \ref{table:symplectic}. This is consistent with  \cite[Theorems 1.2 and 1.8]{LZ22} which we will recall below. By $\Aut^s(X_i')$ being maximal (as indicated by $\checkmark$ in the last column of the table), we mean that $\Aut^s(X_i')$ is not isomorphic to a proper subgroup of $\Aut^s(X)$ for any smooth cubic fourfold $X$.
	In fact, by computer calculations using GAP,  $\Aut^s(X_i')$ is not isomorphic to any proper subgroup of the groups in \cite[Theorem 1.2]{LZ22} if and only if $i\in \{1,4,6,7,10,11,14\}$ (see the ancillary file {\ttfamily43-groups.txt}). 
	\begin{footnotesize}
		\begin{table}[htbp]
			\renewcommand\arraystretch{1.2}
			\begin{tabular}{cccccc}
			 
			\hline
		$i$ & $\Aut(X_i')$ &$|\Aut(X_i')|$ & $\Aut^s(X_i')$ &$|\Aut^s(X_i')|$ & maximal\\
			\hline
			1 & $C_3^5 \rtimes S_6$&174960& $C_3^4\rtimes A_6$&29160&\checkmark \\

			2 & $((C_3 \times (C_3^3\rtimes C_3))\rtimes C_3) \rtimes (C_4 \times C_2)$ &5832&$(C_3\times (C_3^2\rtimes C_3))\rtimes C_2$&486& \\

			3 & $C_8 \times (C_3^2 \rtimes C_2)$&144&$S_3$&6& \\

			4 & $S_5 \times (C_3^2 \rtimes C_2)$ &2160&$A_5\rtimes S_3$&360& \checkmark \\

			5 & $C_{48}$ &48&trivial&1& \\

			6 & $\rm{PSL}(2,11) \times C_3$ &1980&$\PSL(2,11)$&660&\checkmark  \\

			7 & $((C_3 \times(C_3^2\rtimes C_3))\rtimes C_3)\rtimes (C_4^2 \rtimes C_2)$ &7776&$((C_3 \times(C_3^2\rtimes C_3))\rtimes C_3)\rtimes Q_8$&1944&\checkmark  \\ 

			8 & $C_{32}$&32&trivial&1& \\

			9 & $C_{21} \rtimes C_6$ &126&$C_7 \rtimes C_3$&21& \\

			10 & $M_{10}$&720&$M_{10}$&720& \checkmark \\

			11 & $S_7$&5040&$A_7$&2520& \checkmark \\

			12 & $(C_8 \times C_2)\rtimes C_2$&32&$QD_{16}$&16& \\

			13 & ${\rm PSL}(3,2) \rtimes C_2$&336&$\PSL(3,2)$&168& \\

			14 & $\GL(2,3)$&48&$\GL(2,3)$&48& \checkmark\\

			15 & $((C_3 \times C_3)\rtimes Q_8)\rtimes C_3$&216&$(C_3 \times C_3)\rtimes Q_8$&72& \\
			\hline
			\end{tabular}
				 \vspace*{.1in}
			\caption{Symplectic automorphism groups of cubic fourfolds}
			\label{table:symplectic}
		\end{table}
	\end{footnotesize}

	Based on the global Torelli theorem for cubic fourfolds and lattice theory, Laza--Zheng \cite[Theorems 1.2]{LZ22} identified all possible $\Aut^s(X)$ for smooth cubic fourfolds $X$. The following result is a direct consequence of \cite[Theorems 1.2 and 1.8]{LZ22}.
	
	\begin{theorem}\label{thm:LZ} 
	A finite group $G$ can act  faithfully and symplectically on a smooth cubic fourfold if and only if $G$ is isomorphic to a subgroup of one of the following $7$ groups: $$\GL(2,3), C_3^4\rtimes A_6, A_7, ((C_3 \times(C_3^2\rtimes C_3))\rtimes C_3)\rtimes Q_8, M_{10}, {\rm PSL}(2,11), A_5\rtimes S_3.$$ Moreover, for the moduli space $\mathcal{M}_G$ of the smooth cubic fourfolds $X$ with ${\rm Aut}^s(X)$ containing $G$ as a subgroup, the following statements hold:
	\begin{enumerate}
  \item [(1)] If $G=\GL(2,3)$, then ${\rm dim}(\mathcal{M}_G)=1$;
  \item [(2)] If $G=C_3^4\rtimes A_6$ $($resp. $A_7, ((C_3 \times(C_3^2\rtimes C_3))\rtimes C_3)\rtimes Q_8, M_{10}, {\rm PSL}(2,11), A_5\rtimes S_3$$)$, then the cardinality $|\mathcal{M}_G|=1$ $($resp. $2, 1, 2, 1, 1$$)$. 
\end{enumerate}
	\end{theorem}	
	
	\begin{remark}
	Using notations in \cite[Theorem 1.2]{LZ22}, the $7$ groups in Theorem \ref{thm:LZ} are isomorphic to $T_{48}$, $3^4: A_6$, $A_7$, $3^{1+4}:2.2^2$, $M_{10}$, $L_2(11)$, $A_{3,5}$ respectively.  Note that $T_{48}$ is one of the $11$ maximal finite groups acting on K3 surfaces faithfully and symplectically (\cite{Mu88}).
	\end{remark}
	
	Among the $8$ smooth cubic fourfolds with maximal symplectic automorphism groups in Theorem \ref{thm:LZ} (2),  explicit defining equations for $6$ of them are previously known (see \cite[Table 11]{HM19}, \cite[Theorem 1.8]{LZ22}). Let $X^i(A_7)$ and $X^i(M_{10})$ ($i=1,2$) be as in \cite[Theorem 1.8]{LZ22}. Note that $X^1(A_7)\cong X_{11}'$ and $X^1(M_{10})\cong X_{10}'$. To the best of our knowledge, explicit defining equations for $X^2(A_7)$ and $X^2(M_{10})$ are  unknown. As a by-product of our classification of $(4,3)$-groups and $(5,3)$-groups, we solve this open problem.

	\begin{theorem}\label{thm:A7}
	Let $F_{A_7}=x_1^3+x_2^3+x_3^3+\frac{12}{5}x_1x_2x_3+x_1x_4^2+x_2x_5^2+x_3x_6^2+\frac{4\sqrt{15}}{9}x_4x_5x_6$. Then the smooth cubic fourfold $X_{F_{A_7}}$ satisfies $$\Aut^s(X_{F_{A_7}})=\Aut(X_{F_{A_7}})\cong A_7.$$ In particular, $X_{F_{A_7}}$ is isomorphic to $X^2(A_{7})$.
	\end{theorem}
	
	\begin{proof}
	Let $F:=F_{A_7}$. Recall that $F_{16}=\widehat{F}=F+x_7^3$ and $X_{16}=X_{F_{16}}$ (see the example (16) in Subsection \ref{mainex}). By Theorem \ref{thm:Main} and Remark \ref{rem:Aut=G}, we have $\Aut(X_{16})\cong C_3.A_7$ is generated by $A_{X_{16},1}=\begin{pmatrix}A_{X_{16},1}'&0\\0&1\end{pmatrix}$ and $A_{X_{16},2}=\begin{pmatrix}A_{X_{16},2}'&0\\0&1\end{pmatrix}$, where  $A_{X_{16},1}':=\Diag(1,\xi_3,\xi_3^2,-1,\xi_3,-\xi_3^2)$. Then by Lemma \ref{lem:exactseq}, $\Aut(X_F)$ is generated by $[A_{X_{16},1}']$ and $[A_{X_{16},2}']$, which implies that $\Aut(X_F)\cong A_7$.  Since the quotient group $\Aut(X_F)/\Aut^s(X_F)$ is abelian and $A_7$ is a non-abelian simple group, we have $\Aut^s(X_F)=\Aut(X_F)\cong A_7$. By $A_7\cong \Aut^s(X^1(A_7))\subsetneqq \Aut(X^1(A_7))\cong S_7$, the cubic fourfold $X^1(A_7)$ is not isomorphic to $X_F$. Thus, $X_F\cong X^2(A_7)$. 
	\end{proof}
	
	For $M_{10}$, we prove a somewhat stronger result (note that $M_{10}$ contains $A_6$ as a normal subgroup of index $2$) via our approach of classifying $(4,3)$-groups and $(5,3)$-groups.
	
	\begin{theorem}\label{thm:M10}
	Let $X$ be a smooth cubic fourfold. Then the following three statements are equivalent:
	
	\begin{enumerate}
  \item [(1)] $\Aut(X)=\Aut^s(X)\cong M_{10}$;
  \item [(2)] $M_{10}$ is isomorphic to a subgroup of $\Aut(X)$;
  \item [(3)] $X$ is isomorphic to one of the following two smooth cubic fourfolds: $X_{10}'$ and $X_{F_{M_{10}}}$, where $F_{M_{10}}=x_1^3 + 1/1815(1036\xi_{24}^7 - 5800{\xi_4} - 1576\xi_{24}^5 + 2016{\xi_6} + 4180{\xi_8} + 
	3632{\xi_{12}} - 2644{\xi_{24}} - 3939)x_1x_2^2 + 1/605(1028\xi_{24}^7 - 864{\xi_4} - 1468\xi_{24}^5
	- 1448{\xi_6} + 1280\xi_{24}^3 + 3072{\xi_{12}} + 152{\xi_{24}} - 2270)x_1x_3x_4 + 
	1/3993(25574\xi_{24}^7 + 9032{\xi_4} - 20826\xi_{24}^5 - 18220{\xi_6} - 13744\xi_{24}^3 + 
	13592{\xi_{12}} + 22080{\xi_{24}} - 1231)x_2x_3^2 + 1/3993(41818\xi_{24}^7 + 64576{\xi_4} - 
	1314\xi_{24}^5 - 79580{\xi_6} - 60500\xi_{24}^3 - 20552{\xi_{12}} + 70716{\xi_{24}} + 43177)x_2x_4^2 
	+ 1/19965(-16944\xi_{24}^7 - 50216{\xi_4} - 100168\xi_{24}^5 + 192272{\xi_6} - 55224\xi_{24}^3 
	+ 153712{\xi_{12}} - 145288{\xi_{24}} - 22288)x_2x_5x_6 + 1/6655(-20096\xi_{24}^7 + 
	5560{\xi_4} + 22156\xi_{24}^5 + 6216{\xi_6} - 452\xi_{24}^3 - 17296{\xi_{12}} - 11268{\xi_{24}} + 
	13556)x_3x_5^2 + 1/6655(89336\xi_{24}^7 + 48240{\xi_4} - 8948\xi_{24}^5 - 88392{\xi_6} - 
	106476\xi_{24}^3 + 29824{\xi_{12}} + 50884{\xi_{24}} + 70396)x_4x_6^2$. 
\end{enumerate}
	In particular, $X_{F_{M_{10}}}$ is isomorphic to $X^2(M_{10})$.
	\end{theorem}
	
	\begin{proof}
	 The idea is to compute cubic polynomials preserved by $(5,3)$-representations of $C_3$-covering $(5,3)$-groups of $M_{10}$. Let $F:=F_{M_{10}}$. By Theorem \ref{thm:Main}, $M_{10}$ has only one $C_3$-covering $(5,3)$-group $C_3.M_{10}$. By adapting \textbf{Steps 2} and \textbf{3} of Strategy \ref{strategy}, we compute the set $\mathcal{S}_{C_3.M_{10}}$ of all $(5,3)$-representations up to $3$-equivalence. It turns out that  $\mathcal{S}_{C_3.M_{10}}$ contains exactly $2$ representations of $C_3.M_{10}$. An $F_{12}$-lifting of the automorphism group $\Aut(X_{12})$ of $X_{12}$ in the example (12) in Subsection \ref{mainex} corresponds to one representation, say $\rho_1$, in $\mathcal{S}_{C_3.M_{10}}$. Let $\rho_2$ denote the other one in $\mathcal{S}_{C_3.M_{10}}$. It turns out the image $\rho_2(C_3.M_{10})$ can be generated by two matrices $A_1=\begin{pmatrix}A_1'&0\\0&1\end{pmatrix}$ and $A_2=\begin{pmatrix}A_2'&0\\0&1\end{pmatrix}$, where $A_1'=\Diag(1,-1,\xi_4,-\xi_4,\xi_{8}^{7},-\xi_{8})$ and  
	 $A_2'\in \GL(6,\C)$ (the matrix $A_2'$ is a little complicated and it can be found in the ancillary file {\ttfamily Theorem6.15.txt} to \cite{YYZ23}). Then by computation, the set of the homogeneous polynomials of degree $3$ preserved by both $A_1'$ and $A_2'$ is $\{\lambda F\mid \lambda\in \C\}$. The automorphism group $\Aut(X_{\widehat{F}})$ of the smooth cubic fivefold $X_{\widehat{F}}$ contains $\langle[A_1],[A_2] \rangle\cong C_3.M_{10}$, a maximal $(5,3)$-group by Theorem \ref{thm:Main}. Thus, $\Aut(X_{\widehat{F}})=\langle[A_1],[A_2] \rangle$ and $\Aut(X_{F})=\langle [A_1'],[A_2']\rangle\cong M_{10}$. Note that $A_1'(F)=A_2'(F)=F$, ${\rm ord}([A_1'])={\rm ord}(A_1')=8$, ${\rm ord}([A_2'])={\rm ord}(A_2')=4$, and ${\rm det}(A_1')={\rm det}(A_2')=1$. Then by Lemma \ref{lem:symcon}, we have $[A_1'],[A_2']\in\Aut^s(X_{F})=\Aut(X_F)$. Since $\rho_1$ and $\rho_2$  are not $3$-equivalent, it follows that  $X_{F_{12}}=X_{\widehat{F_{10}'}}$ and $X_{\widehat{F}}$ are not isomorphic, which implies that $X_{10}'$ and $X_F$ are not isomorphic. This completes the proof of the theorem.
	\end{proof}
	
	\begin{remark}\label{rem:A7M10}
	The automorphism groups $\Aut(X_{F_{A_7}})$ and $\Aut(X_{F_{M_{10}}})$ have no $F_{A_7}$-lifting and $F_{M_{10}}$-lifting respectively. 	\end{remark}
	
	\begin{remark}\label{rem:classifysym}
	As in the proof of Theorem \ref{thm:M10}, we can compute all $(5,3)$-representations, up to $3$-equivalence, of any $(5,3)$-groups by adapting Strategy \ref{strategy}. In particular, for any $(4,3)$-group $G$, we can determine whether $G$ can act faithfully and symplectically on smooth cubic fourfolds via computing $(5,3)$-representations of $C_3$-covering $(5,3)$-groups of $G$. In this way, we can prove Theorem \ref{thm:LZ} without using the global Torelli theorem for cubic fourfolds.
	\end{remark}

\appendix
\section{Roles of supplementary files}	\label{Appendix}
In this appendix, we briefly explain the role of each supplementary file.
	\begin{enumerate}
		\item [(1)] {\ttfamily Examples4.1.txt}: This file contains the matrix generators of $G_{X_i}\subset \PGL(7,\C)$ for all $i$ except $i=6$, where $X_i$ are the smooth cubic fivefolds in Subsection \ref{mainex}.
		\item [(2)] {\ttfamily Examples6.1.txt}: This file contains the matrix generators of $G_{X_i'}\subset \PGL(6,\C)$ for all $i$ except $i=6$, where $X_i'$ are the smooth cubic fourfolds in Subsection \ref{ex:4folds}.
	    \item [(3)] {\ttfamily 53-groups.txt}:  This file contains lists of all subgroups of the $20$ groups $\Aut(X_i)$ in Theorem \ref{thm:Main}. We use this file in Remark \ref{rem:Aut=G}, Lemma \ref{lem:ab53}, and the ancillary file {\ttfamily GAPsubtest-53groups.txt}. 
		\item [(4)] {\ttfamily 43-groups.txt}: This file contains lists of all subgroups of the $15$ groups $\Aut(X'_i)$ in Theorem \ref{thm:fourfold}. Such subgroups are used in the ancillary file {\ttfamily GAPsubtest-43groups.txt}.
	    \item [(5)] {\ttfamily Theorem5.12.txt}: This file contains representatives of all $3$-equivalence classes of $(5,3)$-representations of the $19$ abelian groups in Theorem \ref{thm:control}. The $(5,3)$-representations in this file and the next file can be computed by hand in principle (see e.g. \cite[Theorem 5.4]{WY20} and Example 5.7), but we use computer algebra for efficiency. 
	    \item [(6)] {\ttfamily Theorem5.13.txt}:  This file contains representatives of all $3$-equivalence classes of $(5,3)$-representations of cyclic groups of primary orders. We use this file in the proof of Theorem \ref{thm:abelian}.
	    \item [(7)] {\ttfamily Theorem6.15.txt}: This file contains the matrix $A_2'$ in the proof of Theorem \ref{thm:M10}.
	    \item [(8)] {\ttfamily GAPsubtest-53groups.txt}: This file contains the GAP codes used in \textbf{Step 1} of Strategy \ref{strategy}.
	    \item [(9)] {\ttfamily GAPsubtest-43groups.txt}: This file contains the GAP codes used in \textbf{Steps 1} and \textbf{2} of Strategy \ref{strategy43}.
	    \item [(10)] {\ttfamily LHsnewLHsnewred.txt}:  
	    This file contains $(5,3)$-characters and special $(5,3)$-characters of non-trivial abelian $(5,3)$-groups required for  \textbf{Steps 2} and  \textbf{3} of Strategy \ref{strategy} and  \textbf{Step 3} of Strategy \ref{strategy43}.
	    \item [(11)] {\ttfamily GAPcodesnewforMath.txt}:  This file contains the GAP codes used in  \textbf{Steps 2} and  \textbf{3} of Strategy \ref{strategy} and  \textbf{Step 3} of Strategy \ref{strategy43}.
	    \item [(12)] {\ttfamily CubicFivefolds.m}: This file contains the Mathematica codes used in  \textbf{Steps 2} and  \textbf{3} of Strategy \ref{strategy} and  \textbf{Step 3} of Strategy \ref{strategy43}.
	    \item [(13)] {\ttfamily Theorem5.1.txt}: This file contains the computation details (for non-solvable groups) in the proof of Theorem \ref{thm:Main}. The computation uses Strategy \ref{strategy}.
	    \item [(14)] {\ttfamily Theorem5.22.txt}: This file contains the computation details (for solvable groups) in the proof of Theorem \ref{thm:solvable}. The computation uses Strategy \ref{strategy}.
	    \item [(15)] {\ttfamily Theorem6.1.txt}: This file contains the computation details in the proof of Theorem \ref{thm:fourfold}. The computation uses Strategy \ref{strategy43}.
	    \item [(16)] {\ttfamily Example-m96-53-groups.txt}: In this file, we take the case $m = 96$ as an example to illustrate how to classify (non-abelian) solvable $(5,3)$-groups by using Strategy \ref{strategy} with the help of computer algebra (GAP, Sage and Mathematica). We use this file in the proof of Theorem \ref{thm:solvable}.
	    \item [(17)] {\ttfamily Example-m16-43-groups.txt}: In this file, we take the case $m = 16$ as an example to illustrate how to classify solvable $(4,3)$-groups by using Strategy \ref{strategy43} with the help of computer algebra (GAP, Sage and Mathematica). We use this file in the proof of Theorem \ref{thm:fourfold}.
	\end{enumerate}

\end{document}